\documentclass[11pt, notitlepage]{article}
\usepackage{amssymb,amsmath,comment}
\catcode`\@=11 \@addtoreset{equation}{section}
\def\thesection{\arabic{section}}

\def\theequation{\thesection.\arabic{equation}}
\catcode`\@=12
\usepackage{colortbl}%
\usepackage{a4wide}
\usepackage{parskip}
\newcommand{\ds} {\displaystyle}
\newcommand{\e}{\epsilon}
\newcommand{\pa} {\partial}
\newcommand{\al} {\alpha}
\newcommand{\ba} {\beta}
\newcommand{\ga} {\gamma}

\newcommand{\Om} {\Omega}
\newcommand{\ra} {\rightarrow}
\newcommand{\ov}{\overline}
\newcommand{\de} {\delta}

\newcommand{\la} {\lambda}
\newcommand{\La} {\Lambda}
\newcommand{\noi} {\noindent}
\newcommand{\na} {\nabla}
\newcommand{\ul} {\underline}
\newcommand{\mb} {\mathbb}

\newcommand{\lra} {\longrightarrow}

\usepackage[all]{xy}
\catcode`\@=11
\def\theequation{\@arabic{\c@section}.\@arabic{\c@equation}}
\catcode`\@=12

\def\QED{\hfill {$\square$}\goodbreak \medskip}

\newtheorem{Theorem}{Theorem}[section]
\newtheorem{Lemma}[Theorem]{Lemma}
\newtheorem{Proposition}[Theorem]{Proposition}

\begin{document}
{\vspace{0.01in}
\title
{Singular elliptic problems with unbalanced growth and critical exponent}

\author{ {\bf Deepak Kumar$\,^{1,}$\footnote{e-mail: {\tt deepak.kr0894@gmail.com}},  Vicen\c tiu D. R\u adulescu$\,^{2,3,}$\footnote{e-mail: {\tt radulescu@inf.ucv.ro}},
		   K. Sreenadh$\,^{1,}$\footnote{e-mail: {\tt sreenadh@maths.iitd.ac.in}}} \\ $^1\,$Department of Mathematics, Indian Institute of Technology Delhi,\\
	Hauz Khaz, New Delhi-110016, India\\
$^2\,$Faculty of Applied Mathematics, AGH University of Science and Technology,\\
 al. Mickiewicza 30, 30-059 Krak\'ow, Poland \\
 $^3\,$Department of Mathematics, University of Craiova, 200585 Craiova, Romania}

\date{}

\maketitle

\begin{abstract}
In this article, we study the existence and multiplicity of solutions of the following $(p,q)$-Laplace equation with singular nonlinearity:
\begin{equation*}
\left\{\begin{array}{rllll}
-\Delta_{p}u-\ba\Delta_{q}u & = \la u^{-\de}+ u^{r-1}, \  u>0, \  \text{ in } \Om \\ u&=0 \quad \text{ on } \pa\Om,
\end{array}
\right.
\end{equation*}
 where $\Om$ is a bounded domain in $\mathbb{R}^n$ with smooth boundary, $1<  q< p<r \leq p^{*}$, where $p^{*}=\ds \frac{np}{n-p}$, $0<\de< 1$, $n> p$ and $\la,\, \ba>0$ are  parameters. We prove existence, multiplicity and regularity of weak solutions of $(P_\la)$ for suitable range of $\la$. We also prove the global existence result for problem $(P_\la)$.

\medskip

\noi \textbf{Key words:}  $(p,q)$-Laplace equation, double-phase energy, Sobolev critical exponent, singular problem.

\medskip

\noi \textit{2010 Mathematics Subject Classification:} 35B65, 35J35, 35J75, 35J92

\end{abstract}

\section{Introduction}
 \noindent In this paper, we are concerned with the study of a nonlinear problem whose features are the following:\\ (i) the presence of several differential operators with different growth, which generates a {\it double phase} associated energy;\\ (ii) the reaction combines the multiple effects generated by a singular term and a nonlinearity with critical growth;\\ (iii) we establish a global existence property, which describes an exhaustive bifurcation picture. Roughly speaking, this result shows that the problem has a solution if and only if the positive parameter associated to the singular nonlinearity is sufficiently small. Summarizing, this paper
 is concerned with the refined qualitative and bifurcation analysis of solutions for a class of {\it singular} problems driven by differential operators with {\it unbalanced growth}.

 We recall in what follows some of the outstanding contributions of the Italian school  to the study of unbalanced integral functionals and double phase problems.  We first refer to the pioneering contributions of Marcellini \cite{marce1,marce2,marce3} who studied lower semicontinuity and regularity properties of minimizers of certain quasiconvex integrals. Problems of this type arise in nonlinear elasticity and are connected with the deformation of an elastic body, cf. Ball \cite{ball1,ball2}. We also refer to Fusco and Sbordone \cite{fusco} for the study of regularity of minima of anisotropic integrals.

In order to recall the roots of double phase problems, let us assume that $\Omega$ is a bounded domain in ${\mathbb R}^n$ ($N\geq 2$) with smooth boundary. If $u:\Omega\to{\mathbb R}^n$ is the displacement and if $Du$ is the $n\times n$  matrix of the deformation gradient, then the total energy can be represented by an integral of the type
\begin{equation}\label{paolo}I(u)=\int_{\Omega} f(x,Du(x))dx,\end{equation}
where the energy function $f=f(x,\xi):\Omega\times{\mathbb R}^{n\times n}\to{\mathbb R}$ is quasiconvex with respect to $\xi$. One of the simplest examples considered by Ball is given by functions $f$ of the type
$$f(\xi)=g(\xi)+h({\rm det}\,\xi),$$
where ${\rm det}\,\xi$ is the determinant of the $n\times n$ matrix $\xi$, and $g$, $h$ are nonnegative convex functions, which satisfy the growth conditions
$$g(\xi)\geq c_1\,|\xi|^p;\quad\lim_{t\to+\infty}h(t)=+\infty,$$
where $c_1$ is a positive constant and $1<p<n$. The condition $p< n$ is necessary to study the existence of equilibrium solutions with cavities, that is, minima of the integral \eqref{paolo} that are discontinuous at one point where a cavity forms; in fact, every $u$ with finite energy belongs to the Sobolev space $W^{1,p}(\Omega,{\mathbb R}^n)$, and thus it is a continuous function if $p>n$. In accordance with these problems arising in nonlinear elasticity, Marcellini \cite{marce1,marce2} considered continuous functions $f=f(x,u)$ with {\it unbalanced growth} that satisfy
$$c_1\,|u|^p\leq |f(x,u)|\leq c_2\,(1+|u|^q)\quad\mbox{for all}\ (x,u)\in\Omega\times{\mathbb R},$$
where $c_1$, $c_2$ are positive constants and $1\leq p\leq q$. Regularity and existence of solutions of elliptic equations with $p,q$--growth conditions were studied in \cite{marce2}.

The study of non-autonomous functionals characterized by the fact that the energy density changes its ellipticity and growth properties according to the point has been continued in a series of remarkable papers by Mingione {\it et al.} \cite{baroni0}--\cite{baroni}, \cite{colombo0}--\cite{colombo1}. These contributions are in relationship with the works of Zhikov \cite{zhikov1}, in order to describe the
behavior of phenomena arising in nonlinear
elasticity.
In fact, Zhikov intended to provide models for strongly anisotropic materials in the contect of homogenisation.
 In particular, Zhikov considered the following model of
functional  in relationship to the Lavrentiev phenomenon:
$$
{\mathcal P}_{p,q}(u) :=\int_{\Omega} (|\nabla u|^p+a(x)|\nabla u|^q)dx,\quad 0\leq a(x)\leq L,\ 1<p<q.
$$
In this functional, the modulating coefficient $a(x)$ dictates the geometry of the composite made by
two differential materials, with hardening exponents $p$ and $q$, respectively.

The functional ${\mathcal P}_{p,q}$ falls in the realm of the so-called functionals with
nonstandard growth conditions of $(p, q)$--type, according to Marcellini's terminology. This is a functional of the type in \eqref{paolo}, where the energy density satisfies
$$|\xi|^p\leq f(x,\xi)\leq  |\xi|^q+1,\quad 1\leq p\leq q.$$

Another significant model example of a functional with $(p,q)$--growth studied by Mingione {\it et al.} is given by
$$u\mapsto \int_{\Omega} |\nabla u|^p\log (1+|\nabla u|)dx,\quad p\geq 1,$$
which is a logarithmic perturbation of the $p$-Dirichlet energy.

General models with $(p,q)$-growth in the context of
geometrically constrained problems have been recently studied by De Filippis \cite{cristina}. This seems to be the first work dealing with $(p,q)$-conditions with manifold constraint. Refined regularity results
are proved
in \cite{cristina},
by using an approximation technique relying on estimates obtained through a careful use of difference quotients.

 The purpose of this paper is to study the existence and multiplicity of solutions of the following $(p,q)$-Laplacian problem
\begin{equation*}
 ( P_\la) \;\left\{\begin{array}{rllll}
   -\Delta_{p}u-\ba\Delta_{q}u & = \la u^{-\de}+ u^{r-1}, \;  u> 0 \text{ in } \Om \\ u&=0 \quad \text{ on } \partial\Om,
\end{array}
\right.
\end{equation*}
\noi where $\Om$ is a bounded domain in $\mb R^n$ with smooth boundary, $1<  q< p<r \leq p^{*}$, with $p^{*}=\ds \frac{np}{n-p}, \ 0< \de\le 1, n> p$ and $\la,\ba>0$ are real parameters. Here, $\Delta_{p}$ is the $p$-Laplace operator, defined as
$\Delta_{p} u= \nabla\cdot(|\nabla u|^{p-2}\nabla u)$.

The differential operator $A_{p,q}:=-\Delta_{p}-\ba\Delta_{q}$ is known as $(p,q)$-Laplacian, which arises from a wide range of important applications such as biophysics \cite{fife}, plasma physics \cite{wilhel}, reaction-diffusion \cite{cherfils}. For more details on applications readers are referred to survey article \cite{maranorcn}.

The study of elliptic equations with singular nonlinearities has drawn the attention of many researchers since the pioneering work of Crandall, Rabinowitz and Tartar \cite{crandall}, where authors studied purely singular problem associated to $-\Delta$ with Dirichlet boundary condition. More generally, the equation of type
  \begin{equation}\label{eqb1}
  	 -\Delta u=\la\;a(x)u^{-\de}+b(x)u^{r-1}, \ \ u>0 \ \mbox{ in }\Om; \ u=0 \ \mbox{ on }\pa\Om,
  \end{equation}
  has been studied in a large number of papers, for instance Coclite and Palmieri \cite{coclite} obtained global existence result for \eqref{eqb1}. Using  Nehari manifold method Yijing, Shaoping and Yiming \cite{yijing} proved existence of two solutions of \eqref{eqb1} when $0<\de<1$ and $r<2^*$. The critical case was dealt by Haitao \cite{haitao} and Hirano, Saccon and Shioji \cite{hirano}. In \cite{haitao},  for $a=1=b$ and $0<\de<1$, Haitao proved global existence of solutions using Perron's method and saddle point theorem while authors in \cite{hirano} used Nehari manifold technique to prove the existence of two solutions.  Adimurthi and Giacomoni \cite{adimurthi} considered problem \eqref{eqb1} for the case $n=2$, $0<\de<3$ with Trudinger-Moser type critical nonlinearities. Ghergu and R\u adulescu \cite{ghergu1, ghergu2} considered singular elliptic equations with gradient term, while Dupaigne, Ghergu and R\u adulescu \cite{ghergu0} studied singular Lane-Emden-Fowler equations with convection and singular potential. For a thorough analysis of semilinear elliptic equations with singular nonlinearities we refer to the monograph by Ghergu and R\u adulescu \cite{ghergu}.

  For general $p$, Giacomoni, Schindler and Tak\'a\v{c} \cite{giacomoni} studied the following singular problem
    \begin{equation*}
    	-\Delta_p u=\la\;u^{-\de}+u^{r-1}, \ \ u>0 \ \mbox{ in }\Om; \ u=0 \ \mbox{ on }\pa\Om,
    \end{equation*}
  where $0<\de<1$ and $1<p<r\le p^*$. In this work authors proved the existence of multiple solutions in $C^{1,\al}(\ov\Om)$ using variational method developed in \cite{garcia} and \cite{figueiredo}. Here, a multiplicity result was obtained for all $p>1$ in the subcritical case and for $p\in\big(\frac{2n}{n+2},2\big]\cup \big(\frac{3n}{n+3},3\big)$ in the critical case. For more work on singular quasilinear elliptic equations we refer to \cite{goyal, mohammed}.

The $(p,q)$-Laplace equation with concave-convex type nonlinearities has been studied by many researchers. For instance,  Yin and Yang \cite{yin} considered the problem
 \begin{align*}
 	-\Delta_{p}u-\Delta_{q} u=|u|^{p^*-2}u+\theta V(x)|u|^{r-2}u+\la f(x,u) \ \mbox{ in }\Om, \ \ u=0 \ \mbox{ on }\pa\Om,
 \end{align*}
 where  $1<r<q<p<n$ and $f(x,u)$ is a subcritical perturbation, to prove multiplicity of solutions using Lusternik-Schnirelman theory, while Gasi\'nski and Papageorgiou \cite{gasinski} obtained the existence of two positive solutions of the problem with concave nonlinearity and carath\'eodory perturbation having subcritical growth (which need not satisfy Ambrosetti-Rabinowitz condition) for the case $2\le q\le p<\infty$. Subsequently,  Marano et. al  \cite{marano} studied this   problem with  Carath\'eodory function having critical growth. Using critical point theory with truncation arguments and comparison principle authors also proved bifurcation type result. For $(p,q)$-Laplacian problem with concave-convex nonlinearities in $\mb R^n$ we refer \cite{huang}.

Regarding the regularity results for weak solutions of $(p,q)$-Laplacian problem we cite the work of  He and Li \cite{He} who proved that weak solutions of
\begin{align*}
	-\Delta_{p}u-\Delta_{q} u=f(x) \ \mbox{ in } \mb R^n
\end{align*}
belong to $L^\infty_\text{loc}(\mb R^n)\cap C^{1,\al}_\text{loc}(\mb R^n)$ for some $\al\in (0,1)$ if $f(x)\in L^{\infty}_{loc}(\mathbb R^n).$ Here, the authors extended their results to equations with general nonlinearity $f(x,u)$ having critical growth with respect to $p$. Furthermore, Baroni, Colombo and Mingione \cite{baroni} proved $C^{1,\al}_\text{loc}(\Om)$ regularity result for minimizers of general double phase equation. For more details on regularity results, interested readers may refer to \cite{colombo,liberman}.

\section{Main results}
Inspired by the above mentioned works, we study in this paper $(p,q)$-Laplacian problem involving singular nonlinearity. Following the approach of \cite{haitao}, which uses Perron's method to obtain a weak solution of singular problem between a sub and super solution, we prove global (for all $\la, \ba$) existence result for $(P_\la)$.  Using Stampacchia's truncation argument and Moser iteration technique for $(p,q)$-Laplacian problem we prove that the weak solutions of $(P_\la)$ are in $L^\infty(\Om)$ and applying some properties established in \cite[Theorem 1]{He} we show the following regularity theorem.
\begin{Theorem}\label{thm4}
	Each weak solution $u$ of problem $(P_\la)$ belongs to $L^\infty(\Om)\cap C^{1,s}_\text{loc}(\Om)$, for some $s\in(0,1)$. Moreover, there exists $\e_\la>0$ such that $u\ge\e_\la\hat{\phi}$ in $\Om$.
\end{Theorem}
To prove our existence results, we use the Nehari manifold technique to obtain minimizers of the energy functional associated to $(P_\la)$ over some subsets of the Nehari manifold.  First we prove that these minimizers are in fact weak solutions of $(P_\la)$. Furthermore, by analyzing the energy levels and identifying the first critical level we  prove multiplicity results for the critical case for all $q<p\le p^*$ by choosing $\ba$ small. We also establish these results if $\ba>0$  and $p\in(2n/(n+2),3)$.

We denote by $\|\cdot\|_{L^m(\Om)}$ the norm on $L^m(\Om)$ by $\|\cdot\|_m$ for $1\le m\le\infty$.
Let $W^{1,p}_0(\Om)$ be the Sobolev space equipped with the norm $\|.\|$ given by
$\|u\|=\|\nabla u\|_p$ for all $u\in W^{1,p}_0(\Om)$ and $S$ be the Sobolev constant defined as
\begin{align*}
	S=\inf_{u\in W^{1,p}_0(\Om)\setminus\{0\}} \frac{\|u\|^p}{\|u\|_{p^*}^p}.
\end{align*}
We denote by $\la_1(q,\ba)>0$  the first eigenvalue of $-\beta \Delta_{q}$ with homogeneous Dirichlet boundary condition:
  \begin{align*}
  	-\beta \Delta_{q} u=\la u^{q-1}, \ \ u>0 \ \mbox{ in }\Om \ \text{ and }u=0 \ \mbox{ in }\pa\Om.
  \end{align*}
We say that $u\in W^{1,p}_0(\Om)$ is a weak solution of problem $(P_\la)$ if $u> 0$ a.e. in $\Om$ and
\begin{equation}\label{eqbw}
	\int_\Om (|\na u|^{p-2}\na u\na\phi-\ba|\na u|^{q-2}\na u\na\phi-\la u^{-\de}\phi-u^{r-1}\phi)dx=0 \quad \mbox{for all }\phi\in W^{1,p}_0(\Om).
\end{equation}
The Euler functional associated to the problem $(P_\la)$, $I_\la:W^{1,p}_0(\Om)\to \mb R$ is defined as
\begin{align*}
	I_\la(u)=\frac{1}{p}\int_\Om |\na u|^p+\ba\frac{1}{q}\int_{\Om}|\na u|^q-\frac{\la}{1-\de}\int_\Om |u|^{1-\de}~dx-\frac{1}{r}\int_{\Om}|u|^r~dx.
\end{align*}
Set $W^{1,p}_0(\Om)_+:=\{u\in W^{1,p}_0(\Om): u\ge 0 \ \mbox{a.e. in }\Om\}.$
We show the following existence and multiplicity results:
\begin{Theorem}\label{thm1}
	Let $r < p^{*}$. Then there exists $\la_*>0 $ such that  for all $\ba>0$ and $\la \in(0, \la_*)$, $(P_\la)$ has at least two solutions.
\end{Theorem}

 \begin{Theorem}\label{thm2}
 	Let $r=p^*$, then there exists $\La>0$ such that for all $\la\in(0,\La)$ and $\ba>0$ problem $(P_\la)$ admits at least one solution.
 \end{Theorem}
In the critical case, we have the following multiplicity results for ''small $\ba$'' and with no further restriction on $q$:
\begin{Theorem}\label{thm3}
   Let $r=p^*$, then there exist positive constants $\ba_*, \La_0$ and $\ba_0$ such that problem $(P_\la)$ has at least two solutions in each of the following cases:\begin{enumerate}
   \item[(i)] for all $\la\in(0,\La)$ and $\ba\in(0,\ba_*)$, when $\frac{2n}{n+2}<p<3$,
   \item[(ii)] for all $\la\in(0,\La_{0})$ and $\ba\in(0, \ba_{0})$, when $p\in(1,\frac{2n}{n+2}\big]\cup[3,n)$. \end{enumerate}
	\end{Theorem}
 We also show the following multiplicity result for "all $\ba>0$" but with restriction on $p$ and $q$:
  \begin{Theorem}\label{thm6}
Let $p\in\big(\frac{2n}{n+2},3)$ and $r=p^*$, then there exists at least two solutions of $(P_\la)$ for all $\la\in(0,\La)$ and $\ba>0$ in each of the following cases:
	\begin{enumerate}
		\item[(1)]  $\max\{p-1,1\}<q<\frac{n(p-1)}{n-1}$,
		\item[(2)] $\frac{n(p-1)}{n-1}<q<\frac{n(p-1)+p}{n}$.
	\end{enumerate}
\end{Theorem}
Finally, we have the following global existence result for $(P_\la)$.
\begin{Theorem}\label{thm5}
	There exists $\La^*>0$ such that problem $(P_\la)$ has a solution for all $\la\in(0,\La^*]$ and no solution if $\la>\La^*$.
\end{Theorem}

The analysis developed in this paper shows that the singular nonlinearity $u^{-\delta}$ (with $0<\delta<1$) in problem $(P_\la)$ can be replaced by a general nonlinearity $g\in C^1(0,\infty)$ which is positive, decreasing, satisfying $\lim_{t\rightarrow 0^+}g(t)=+\infty$ and such that
\begin{equation}\label{ko}\int_0^1g(t)dt<+\infty.\end{equation}
We observe that \eqref{ko} implies the following Keller-Osserman type condition around the origin:
\begin{equation}\label{ko1}
\int_0^1\left(\int_0^t g(s)ds\right)^{-1/2}dt<+\infty.\end{equation}
As proved by B\'enilan, Brezis and Crandall \cite{beni}, condition \eqref{ko1} is equivalent to the {\it property of compact support}, that is, for every $h\in L^1({\mathbb R}^n)$ with compact support, there exists a unique $u\in W^{1,1}({\mathbb R}^n)$ with compact support such that $\Delta u\in L^1({\mathbb R}^n)$ and
$$-\Delta u+g(u)=h\quad\mbox{in}\ {\mathbb R}^n.$$

The paper is organized as follows: In Section 2, we prove existence result for purely singular problem associated with $(P_\la)$ and prove Theorem \ref{thm4}. In Section 3, we study the fibering maps and Nehari manifold associated with $(P_\la)$. We prove some technical results here. In Section 4, we prove Theorem \ref{thm1}. In Section 5, we prove Theorem \ref{thm2}, Theorem \ref{thm3} and Theorem \ref{thm6}. In Section 6, we give proof of Theorem \ref{thm5}.

\section{A regularity result}
In this Section we study regularity of weak solutions of problem $(P_\la)$ and obtain a weak solution of purely singular problem associated to $(P_\la)$.
  \begin{Lemma}\label{lemm13}\cite{stampacchia}
	Let $\psi$ be a function such that \begin{enumerate}
		\item[(1)] $\psi(t)\ge 0$,
		\item[(2)] $\psi$ is non-increasing,
		\item[(3)] if $h>k>k_0$, then $\psi(h)\le \frac{C}{(h-k)^\rho}(\psi(k))^\ga$, for some $\ga>1$.
	\end{enumerate}
	Then,
	  $\psi(k_0+d)=0$, where $d^\rho:=C\; 2^\frac{\rho\ga}{\ga-1}\big(\psi(k_0)\big)^{\ga-1}$.
\end{Lemma}

\begin{Lemma}\label{lemm14}
	Each weak solution $u$ of $(P_\la)$ belongs to $L^{\infty}(\Om)$.
\end{Lemma}
\begin{proof}
	Let $u$ be a weak solution of $(P_\la)$.
	We follow approach of \cite[Lemma A.6]{giacomoni} to prove
	\begin{equation}\label{eqb11}
	  \int_{\Om}\big(|\na (u-1)_+|^{p-2}\na (u-1)_+ +\ba|\na (u-1)_+|^{q-2}\na (u-1)_+\big)\na w\le C\int_{\Om} (1+(u-1)_+^{p^*-1})w,
	\end{equation}
	for every $w\in W^{1,p}_0(\Om)_+$. Let $\varphi:\mb R\to [0,1]$ be a $C^1$ cut-off function such that $\varphi(t)=0$ for $t\le 0$, $\varphi^\prime(t)\ge 0$ for $0\le t\le 1$ and $\varphi(t)=1$ for $t\ge 1$. For any $\e>0$, define $\varphi_\e(t):=\varphi\big(\frac{t-1}{\e}\big)$ for $t\in \mb R$. Hence $\varphi_\e( u)\in W^{1,p}_0(\Om)$ with $\na(\varphi_\e(u))=(\varphi_\e^\prime(u))\na u.$ Let $w\in C_c^\infty(\Om)$ be such that $w\ge 0$, then using $\varphi_\e(u)w$ as test function in \eqref{eqbw}, we obtain
	\begin{align*}
		\int_\Om \big(|\na u|^{p-2}\na u\na\big(\varphi_\e(u)w\big)+\ba|\na u|^{q-2}\na u\na\big(\varphi_\e(u)w\big)\big)dx=\int_{\Om}\big(\la u^{-\de}+u^{r-1}\big)\varphi_\e(u)wdx.
	\end{align*}
  Hence, \begin{align*}
  	\int_\Om &\big(|\na u|^p+\ba |\na u|^q\big)\varphi_\e^\prime(u)wdx+\int_\Om \big(|\na u|^{p-2}+\ba |\na u|^{q-2}\big)(\na u\cdot\na w)\varphi_\e (u)dx \\ &=\int_{\Om}\big(\la u^{-\de}+u^{r-1}\big)\varphi_\e (u)wdx,
  \end{align*}
  using the fact $\varphi_\e^\prime(u)\ge 0$, above equation yields
  \begin{align*}
  	\int_\Om \big(|\na u|^{p-2}+\ba |\na u|^{q-2}\big)(\na u\cdot\na w) \varphi_\e(u)dx \le \int_{\Om}\big(\la u^{-\de}+u^{r-1}\big)\varphi_\e(u)wdx.
  \end{align*}
  Letting $\e\ra 0^+$, we see that there exists a constant $C>0$ which may depend on $\la$, such that
    \begin{align*}
   \int_{\Om\cap\{u\ge 1\}} \big(|\na u|^{p-2}+\ba |\na u|^{q-2}\big)(\na u\cdot\na w)~dx&\le \int_{\Om\cap\{u\ge 1\}}\big(\la u^{-\de}+u^{r-1}\big)w~dx\\
    &\le C\int_{\Om} (1+(u-1)_+^{p^*-1})w~dx,
   \end{align*}
   this gives us
   \begin{equation*}
   \int_{\Om}\big(|\na (u-1)_+|^{p-2}\na (u-1)_+ +\ba|\na (u-1)_+|^{q-2}\na (u-1)_+\big)\na w\le C\int_{\Om} (1+(u-1)_+^{p^*-1})w,
   \end{equation*}
   for every $w\in C_c^\infty(\Om)$ with $w\ge 0$. Proof of \eqref{eqb11} can be completed proceeding similar as in the proof of \cite[Lemma A.5]{giacomoni}.
   By the  proof of \cite[Theorem 2]{He}, we get $(u-1)_+\in L^{(p^*)^2/p}_\text{loc}(\Om)$ and since $\Om$ is a bounded domain, we conclude that  $(u-1)_+\in L^{(p^*)^2/p}(\Om)$. Repeating the arguments used in proof of \cite[Theorem 2]{He} and using interpolation identity for $L^m$ spaces one can show that $(u-1)_+\in L^m(\Om)$ for all $1\le m<\infty$. Now we will prove that $u\in L^\infty(\Om)$. Set $\bar{u}=(u-1)_+\ge 0$. Consider the truncation function $T_k(s):=(s-k)\chi_{[k,\infty)}$, for $k>0$, which was introduced in \cite{stampacchia}. Let $\Om_k:=\{x\in\Om: \bar{u}(x)\ge k\}$, then taking $T_k(\bar{u})$ as a test function in \eqref{eqb11}, we obtain
	\begin{equation}\label{eqb12}
		\int_{\Om}\big(|\na\bar{u}|^{p-2}\na\bar{u} +\ba|\na\bar{u}|^{q-2}\na\bar{u}\big)\na T_k(\bar{u})\le C\int_{\Om} (1+\bar{u}^{p^*-1})T_k(\bar{u}).
	\end{equation}
	Let $\al>0$ be fixed (to be specified latter). Using the fact $|T_k(\bar{u})|\le \bar{u}$ and $\bar{u}\in L^{\al}(\Om)$ together with H\"older inequality and Sobolev embeddings, we deduce that
	  \begin{align*}
	  	C\int_{\Om} (1+\bar{u}^{p^*-1})T_k(\bar{u}) ~dx\le C\int_{\Om_k}(|\bar{u}|+|\bar{u}|^{p^*})~dx&\le C\left(\int_{\Om_k}(|\bar{u}|+|\bar{u}|^{p^*})^\frac{\al}{p^*}\right)^\frac{p^*}{\al} |\Om_k|^{1-\frac{p^*}{\al}} \\
	  	&\le C_1|\Om_k|^{1-\frac{p^*}{\al}},
	  \end{align*}
	and
	\begin{align*}
		\int_{\Om}|\na\bar{u}|^{p-2}\na\bar{u}\;\na T_k(\bar{u})\ge C_2\int_{\Om_k}|\na T_k(\bar{u})|^p\ge C_3\left(\int_{\Om_k} |T_k(\bar{u})|^{p^*}\right)^\frac{p}{p^*}.
	\end{align*}
	Similarly we can show that
	 \begin{align*}
	 	\int_{\Om}|\na\bar{u}|^{q-2}\na\bar{u}\;\na T_k(\bar{u})~dx\ge 0,
	 \end{align*}
	and for $0<k<h$
	  \begin{align*}
	  	\int_{\Om_k} |T_k(\bar{u})|^{p^*}\ge (h-k)^{p^*}|\Om_h|,
	  \end{align*}
due to the fact $\Om_h\subset\Om_k$. Using all these informations in \eqref{eqb12}, we obtain
	  \begin{align*}
	  	\psi(h)\le\frac{C_4}{(h-k)^{p^*}}(\psi(k))^{\big(1-\frac{p^*}{\al}\big)\frac{p^*}{p}},
	  \end{align*}
	where $\psi(j)=|\Om_j|$ for $j\ge 0$. We choose $\al>0$ such that $\big(1-\frac{p^*}{\al}\big)\frac{p^*}{p}>1$. Then from Lemma \ref{lemm13}, for $k_0=0$, $\rho=p^*$, $\ga=\big(1-\frac{p^*}{\al}\big)\frac{p^*}{p}>1$ and $C= C_4$, we get
	   $\psi(d)=0$, where $d^{p^*}=C_4\; 2^\frac{p^*\ga}{\ga-1}|\Om|^{\ga-1}$, that is $|\Om_d|=|\{x\in\Om: \bar{u}\ge d\}|=0$. Hence, $\bar{u}\in L^\infty(\Om)$ and because $u$ is non-negative we get $u\in L^\infty(\Om)$.
	\QED
\end{proof}

Let us fix $\hat{\la}>\la_1(q,\ba)$, then from \cite[Theorem 2]{tanaka}, we know that the problem
\begin{align*}
-\Delta_{p} u-\ba\Delta_{q} u =\hat{\la} |u|^{q-2}u \ \mbox{ in } \Om, \quad u=0 \ \mbox{on }\pa\Om,
\end{align*}
has a positive solution $\hat{\phi}\in C^{1,\sigma}(\bar{\Om})$ for some $\sigma\in(0,1)$. Now we consider purely singular problem associated to $(P_\la)$,
\begin{equation*}
(S_\la)\quad \left\{
-\Delta_{p}u-\ba\Delta_{q}u  = \la u^{-\de}, \ u>0 \;  \text{ in } \Om, \;  u=0 \quad \text{ on } \pa\Om.
\right.
\end{equation*}
\begin{Lemma} \label{lemm10}
	Problem $(S_\la)$ has a unique solution $\ul u_\la$ in $W^{1,p}_0(\Om)$ for all $\la>0$. Moreover, $\ul u_\la\ge\e_\la\hat{\phi}$ a.e. in $\Om$ for some $\e_\la>0$.
\end{Lemma}
\begin{proof}
	The energy functional corresponding to $(S_\la)$ is given by
	\begin{align*}
	  \tilde{I_\la}(u):=\frac{1}{p}\int_\Om |\na u|^p+\frac{\ba}{q}\int_{\Om}|\na u|^q -\frac{\la}{1-\de}\int_\Om u_+^{1-\de}~dx, \ \  u\in W^{1,p}_0(\Om).
	\end{align*}
	It is easy to verify that $\tilde{I_\la}$ is coercive and weakly lower semicontinuous on $W^{1,p}_0(\Om)$. Therefore,  $\tilde{I_\la}$ has a global minimizer $\ul u_\la\in W^{1,p}_0(\Om)$. Moreover, due to the fact $\tilde{I_\la}(0)=0>\tilde{I_\la}(\e\hat{\phi})$ for sufficiently small $\e>0$, we have $\ul u_\la\neq 0$ in $\Om$ and hence without loss of generality we may assume $\ul u_\la\ge 0$.
	Next we will show that $\ul u_\la\ge \e\hat{\phi}$ a.e. in $\Om$ for some constant $\e>0$. First we observe that G\v ateaux derivative $\tilde{I_\la}^\prime(\e\hat{\phi})$ of $\tilde{I_\la}$ exists at $\e\hat{\phi}$ and satisfies weakly
	\begin{equation}\label{eqb61}
	\begin{aligned}
	\tilde{I_\la}^\prime(\e\hat{\phi})=-\Delta_{p}(\e\hat{\phi})-\ba\Delta_{q}(\e\hat{\phi})-\la (\e\hat{\phi})^{-\de}
	&=\hat{\la}(\e\hat{\phi})^{q-1}-\la(\e\hat{\phi})^{-\de}\\&=(\e\hat{\phi})^{-\de}\big(\hat{\la}(\e\hat{\phi})^{q-1+\de}-\la\big)\\
	&\le -\frac{\la}{2}(\e\hat{\phi})^{-\de}<0
	\end{aligned}
	\end{equation}
	whenever $\e>0$ is small enough, say, $0<\e<\e_\la$. Suppose the function $v=(\ul u_\la-\e_\la\hat{\phi})_-=(\e_\la\hat{\phi}-\ul u_\la)_+$ does not vanish identically on some positive measure subset of $\Om$. Set $\xi(t):=\tilde{I_\la}(\ul u_\la+tv)$ for $t\ge 0$. We note that $\xi$ is convex and $\xi(t)\ge\xi(0)$ for all $t>0$. Furthermore, due to the fact $\ul u_\la+tv\ge\max\{\ul u_\la,t\e_\la\hat{\phi}\}\ge t\e_\la\hat{\phi}$ for $t>0$, the G\v ateaux derivative $\tilde{I_\la}^\prime(\ul u_\la+tv)$ of $\tilde{I_\la}$ exists at $\ul u_\la+tv$  and
	\begin{align*}
	\xi^\prime(t)=\langle \tilde{I_\la}^\prime(\ul u_\la+tv),v\rangle
	\end{align*}
	for all $t>0$. Due to convexity of $\xi$ and the fact $\xi(t)\ge\xi(0)$ for all $t>0$ we see that $\xi^\prime$ is nonnegative and nondecreasing. Therefore, with the help of \eqref{eqb61}, we have
	\begin{align*}
	0\le\xi^\prime(1)&=\int_{\Om}\big(|\na (\ul u_\la+v)|^{p-2}\na (\ul u_\la+v)\na v+\ba|\na (\ul u_\la+v)|^{q-2}\na (\ul u_\la+v)\na v\big)dx\\
	&\qquad-\la\int_{\Om} (\ul u_\la+v)^{-\de}v~dx\\
	&\le -\frac{\la}{2}\int_{\{v>0\}}(\e_\la\hat{\phi})^{-\de}v<0,
	\end{align*}
	which is a contradiction. Thus $v\equiv 0$ in $\Om$, that is, $\ul u_\la\ge \e_\la\hat{\phi}$ a.e. in $\Om$. Since $\tilde{I_\la}$ is strictly convex on $W^{1,p}_0(\Om)_+$, we conclude that such a $\ul u_\la$ is unique.\QED
    \end{proof}

\begin{Lemma}\label{lemm18}
	Let $\ul u_\la$ be the solution of problem $(S_\la)$ and $\bar{u}$ be any weak supersolution (or solution) of $(P_\la)$, then the following comparison principle holds
	  $$ \ul u_\la\le \bar{u} \ \mbox{ a.e. in } \Om.$$
\end{Lemma}
\begin{proof}
 Since $\bar{u}$ is a weak supersolution of $(P_\la)$, we have
  \begin{equation}\label{eqb73}
  	 \int_\Om\big(|\na\bar{ u}|^{p-2}\na \bar{u}+\ba|\na \bar{u}|^{q-2}\na \bar{u}\big)\na\phi ~dx\ge \int_\Om (\la {\bar u}^{-\de}+ {\bar u}^{r-1})\phi~dx,
  \end{equation}
	for all $\phi\in W^{1,p}_0(\Om)$ with $\phi\ge 0$. Let $\eta$ be a smooth function such that $\eta(t)=1$, for $t\ge 1$, $\eta(t)=0$, for $t\le 0$ and $\eta^\prime(t)\ge 0$ for $t\ge 0$. For $\e>0$, set $\eta_\e(t)=\eta\big(\frac{t}{\e}\big)$, then using $\eta_\e(\ul u_\la-\bar{u})$ as a test function in \eqref{eqb73} and in the weak formulation of $(S_\la)$, we deduce that
	\begin{equation}\label{eqb72}
	  \begin{aligned}
	   \int_\Om &\big(|\na \bar{u}|^{p-2}\na \bar{u}-|\na\ul u_\la|^{p-2}\na\ul u_\la\big) \na\eta_\e(\ul u_\la-\bar{u})~dx\\
	   & \quad+\ba \int_\Om \big(|\na \bar{u}|^{q-2}\na \bar{u}-|\na\ul u_\la|^{q-2}\na\ul u_\la\big) \na\eta_\e(\ul u_\la-\bar{u})~dx \\
	   &\ge \int_\Om\big(\la\bar{u}^{-\de}+\bar{u}^{r-1}-\la\ul u_\la^{-\de}\big)\eta_\e(\ul u_\la-\bar{u})~dx \ge\la \int_\Om\big(\bar{u}^{-\de}-\ul u_\la^{-\de}\big)\eta_\e(\ul u_\la-\bar {u})~dx.
	 \end{aligned}
	\end{equation}
	Using the fact $\na\eta_\e(\ul u_\la-\bar{u})=\eta_\e^\prime(\ul u_\la-\bar{u}) \na(\ul u_\la-\bar{u})$ and $\eta^\prime(t)\ge 0$, we get
	\begin{align*}
	    \int_\Om \big(|\na\bar{u}|^{p-2}\na\bar{u} &-|\na\ul u_\la|^{p-2}\na\ul u_\la\big) \na\eta_\e(\ul u_\la-\bar{u})\\
	    &=-\int_\Om \big(|\na\ul u_\la|^{p-2}\na\ul u_\la -|\na\bar{u}|^{p-2}\na\bar{u}\big) \na(\ul u_\la-\bar{u})\eta_\e^\prime(\ul u_\la-\bar{u})\\
	     &\le -C_p\begin{cases}
 	      \ds\int_\Om |\na(\ul u_\la-\bar{u})|^p \eta_\e^\prime(\ul u_\la-\bar{u}), \ \mbox{ if }p\ge 2,\\
	      \ds\int_\Om \frac{|\na(\ul u_\la-\bar{u})|^2}{\big(|\na\ul u_\la|+|\na\bar{u}|\big)^{2-p}}\eta_\e^\prime(\ul u_\la-\bar{u}), \ \mbox{ if }1<p< 2
	  \end{cases}\\
	   &\le 0.
	\end{align*}
Here we used the inequality: there exists a constant $C_p>0$ such that for $a,b\in\mb R^n$,
 \begin{align*}
 |a|^{p-2}|a|-|b|^{p-2}|b|\ge C_p\begin{cases}
 |a-b|^p, \qquad   \mbox{ if }p\ge 2 \\
 \frac{|a-b|^2}{(|a|+|b|)^{2-p}}, \quad \mbox{ if } 1<p<2.
 \end{cases}
 \end{align*}
  Similar result holds for the other term on LHS of \eqref{eqb72}, thereby we infer
	\begin{align*}
	  \la \int_\Om\big(\bar{u}^{-\de}-\ul u_\la^{-\de}\big)\eta_\e(\ul u_\la-\bar{u})\le 0.
	\end{align*}
	Letting $\e\ra 0$, we obtain
	\begin{align*}
	\int_{\{\ul u_\la> \bar{u}\}}(\bar{u}^{-\de}-\ul u_\la^{-\de})dx\le 0,
	\end{align*}
	which implies that $|\{\ul u_\la>\bar{u} \}|=0$, therefore $\ul u_\la\le \bar{u}$ a.e. in $\Om$. \QED
\end{proof}

\begin{Lemma}\label{lemm15}
	For each weak solution $u$ of $(P_\la)$, $|\na u|\in L^{\infty}_\text{loc}(\Om)$ and there exists $s\in(0,1)$ such that $u\in C^{1,s}_\text{loc}(\Om)$.
\end{Lemma}
\begin{proof}
	Let $f(x):=f(x,u(x))=u^{-\de}+u^{r-1}$. Then it is easy to see that $f\in L^\infty_\text{loc}(\Om)$, therefore the result follows from \cite[Theorem 1]{He}.\QED
\end{proof}

\textbf{Proof of Theorem \ref{thm4}}:
Proof of the regularity results follow from Lemmas \ref{lemm14} and \ref{lemm15}. Using Lemmas \ref{lemm10} and \ref{lemm18} we complete proof of the Theorem.\QED

 \section{The Nehari manifold}
 It is easy to verify that the energy functional $I_\la$ is not bounded below on $W^{1,p}_0(\Om)$. For each $u\in W^{1,p}_0(\Om)\setminus\{0\}$ define fibering map $J_u:\mb R_+\to\mb R$ associated to the energy functional $I_\la$ as
 $ J_u(t)=I_\la(tu)$ that is,
 \begin{align}
 	J_u(t)&=\frac{t^p}{p}\int_\Om |\na u|^p+\ba\frac{t^q}{q}\int_{\Om}|\na u|^q -\la\frac{t^{1-\de}}{1-\de}\int_\Om |u|^{1-\de}~dx-\frac{t^r}{r}\int_\Om|u|^r~dx \nonumber \\
 	J_u^\prime(t)&=t^{p-1}\int_\Om |\na u|^p+\ba t^{q-1}\int_{\Om}|\na u|^q-\la t^{-\de}\int_\Om |u|^{1-\de}~dx-t^{r-1}\int_\Om|u|^r~dx \label{eqb21} \\
 	J_u^{\prime\prime}(t)&=(p-1)t^{p-2}\int_\Om |\na u|^p+\ba (q-1)t^{q-2}\int_{\Om}|\na u|^q+\la\de t^{-\de-1}\int_\Om |u|^{1-\de}~dx \label{eqb22}\\ &\quad-(r-1)t^{r-2}\int_\Om|u|^r~dx.\nonumber
 \end{align}
 We define the Nehari manifold $N_\la$ associated to problem $(P_\la)$ as
 \begin{align*}
 N_\la=\{u\in W^{1,p}_0(\Om): u\neq 0, \ J^\prime_u(1)=0\}.
 \end{align*}

 \begin{Lemma}
 	The functional $I_\la$ is coercive and bounded below on $N_\la$.
 \end{Lemma}
\begin{proof}
  Let $u\in N_\la$. Then, using H\"older inequality and Sobolev embedding theorems, we deduce that
	\begin{align*}
	  I_\la(u)
	   &=\left(\frac{1}{p}-\frac{1}{r}\right)\int_\Om |\na u|^p+ \ba\left(\frac{1}{q}-\frac{1}{r}\right)\int_\Om|\na u|^q -\left(\frac{1}{1-\de}-\frac{1}{r}\right)\int_\Om |u|^{1-\de}~dx\\
	   &\ge \left(\frac{1}{p}-\frac{1}{r}\right) \|u\|^p -\left(\frac{1}{1-\de}-\frac{1}{r}\right)C \|u\|^{1-\de}.
	\end{align*}
  Since $1-\de<p$, it follows that $I_\la$ is coercive and bounded below in this case.
	\QED
\end{proof}
We split $N_\la$ into points of maxima, points of minima and inflection points, that is
 \begin{align*}
	N_\la^{\pm}=\left\{u\in N_{\la}: J_{u}^{\prime\prime}(1)\gtrless 0\right\},\; \text{and } N_\la^0=\left\{u\in N_{\la}: J_{u}^{\prime\prime}(1) =0\right\}.
 \end{align*}

  \noindent Define \begin{align*}
 \theta_{\la} := \inf\{ I_{\la}(u) \ | \ u \in {N}_{\la}\} \quad \mbox{and}\ \ \theta_{\la}^\pm := \inf\{ I_{\la}(u) \ | \ u \in {N}_{\la}^\pm\}.
 \end{align*}

 \begin{Lemma}
 	There exists $\la_*>0$ such that for all $\la\in (0,\la_*)$, $N_\la^0=\emptyset$.
 \end{Lemma}
 \begin{proof}
   Suppose $u \in {N}_{\la}^{0}$, then \eqref{eqb21} and \eqref{eqb22}, implies that
  	 \begin{eqnarray}\label{eqb23}
 	   (p-1+\de)\|\na u\|_p^p + \ba (q-1+\de)\|\na u\|_q^q &=& (r-1+\de) \|u\|_r^r dx,\\\label{eqb24}
 	   (r-p)\|\na u\|_p^p+\ba (r-q)\|\na u\|_q^q &=& (r-1+\de)\la\int_\Omega |u|^{1-\de} dx.
 	  \end{eqnarray}
   Define $ E_{\la}: N_{\la} \rightarrow \mb R$ as
 	 \begin{equation*}
 	     E_{\la}(u) = \frac{(r-p)\|\na u\|_p^p +\ba (r-q)\|\na u\|_q^q}{(r-1+\de)} - \la\int_\Om |u|^{1-\de} dx.
 	 \end{equation*}
   Then with the help of \eqref{eqb24} we infer that $E_{\la}(u) = 0$ for all $u\in N_{\la}^{0}$.  Moreover,
 	 \begin{align*}
 	    E_{\la}(u) & \geq \left(\frac{r-p}{r-1+\de}\right)\| u\|^p -\la\int_\Om |u|^{1-\de} dx\\
 	    & \geq \left(\frac{r-p}{r-1+\de}\right)\| u\|^p - \la S^{-\frac{1-\de}{p}} |\Om|^{1-\frac{1-\de}{p^*}} \|u\|^{1-\de}\\
 	    & \geq \| u\|^{1-\de} \left[\left(\frac{r-p}{r-1+\de}\right)\|u\|^{p-1+\de} -\la S^{-\frac{1-\de}{p}} |\Om|^{1-\frac{1-\de}{p^*}}\right].
 	 \end{align*}
   With the help of \eqref{eqb23} and Sobolev embeddings, we have
 	\begin{equation*}
 	   \|u\|\geq \left(\frac{(p-1+\de) S^\frac{r}{p}}{(r-1+\de)|\Om|^{1-\frac{r}{p^*}}}\right)^{\frac{1}{r-p}},
 	\end{equation*}
  as a result
 	\begin{equation*}
 	   E_{\la}(u) \geq \|u\|^{1-\de}\left(\left(\frac{r-p}{r-1+\de}\right)\left(\frac{(p-1+\de) S^\frac{r}{p}}{(r-1+\de)|\Om|^{1-\frac{r}{p^*}}}\right)^{\frac{p-1+\de}{r-p}} -\la S^{-\frac{1-\de}{p}} |\Om|^{1-\frac{1-\de}{p^*}}\right).  \end{equation*}
   Set
 	\begin{equation*}
 	  \la_*:=\;\left(\frac{(r-p)S^\frac{1-\de}{p}}{(r-1+\de)|\Om|^{1-\frac{1-\de}{p^*}}}\right)\left(\frac{(p-1+\de)S^\frac{r}{p}}{(r-1+\de)|\Om|^{1-\frac{r}{p^*}}}\right)^{\frac{p-1+\de}{r-p}} >0,
 	\end{equation*}
  \noi then $E_{\la}(u)>0$ for all $\la\in (0, \la_*)$ and $ u \in{N}_{\la}^{0}$, which contradicts the fact that $ E_\la(u)=0$ for all  $u\in N_\la^0$. This proves the lemma.
 	\QED
 \end{proof}

 For fixed $u\in X$, define $ M_{u}: \mb R^{+} \lra \mb R$ as
\begin{equation*}
M_{u}(t)= t^{p-1+\de}\|\na u\|_p^p +\ba \; t^{q-1+\de}\|\na u\|_q^q-  t^{r-1+\de}\int_{\Om} |u|^{r}dx.
\end{equation*}
Then,
\begin{align*}
M_{u}^{\prime}(t)&= (p-1+\de)t^{p+\de-2}\|\na u\|_p^p +\ba \; (q-1+\de)t^{q+\de-2}\|\na u\|_q^q - (r-1+\de)t^{r+\de-2}\int_{\Om}|u|^{r} dx.
\end{align*}
We notice that for $t>0$, $tu\in N_{\la}$ if and only if $t$ is a solution of
$M_{u}(t)={\la} \ds\int_{\Om}|u|^{1-\de} dx$ and if $tu\in N_\la$, then $J_{tu}^{\prime \prime}(1)=t^{-\de}M_{u}^{\prime}(t)$.
 We claim that there exists unique $t_{max}>0$ such that $M_u^\prime(t_{max})=0$. We have
\begin{align*}
M_u^\prime(t)= &t^{q+\de-2}G_u(t),
\end{align*}
where $G_u(t)=(p-1+\de)t^{p-q}\|\na u\|_p^p +\ba \; (q-1+\de)\|\na u\|_q^q- (r-1+\de)t^{r-q}\ds\int_{\Om}|u|^{r} dx$, then to prove the claim it is enough to show the existence of unique $t_{max}>0$ satisfying $G_u(t_{max})=0$. Define  $H_u(t)= (r-1+\de)\; t^{r-q}\ds\int_{\Om} |u|^{r} dx -(p-1+\de)t^{p-q}\|\na u\|_p^p$, then $H_u(t_{max})- \ba\;(q-1+\de)\|\na u\|_q^q= -G_u(t)$. It is easy to see  $H_u(t)<0$ for $t$ small enough, $H_u(t)\ra\infty$ as $t\ra\infty$. Hence, there exists unique $t_*>0$ such that $H_u(t_*)=0$. Therefore, there exists unique $t_{max}>t_*>0$ such that $H_u(t_{max})= \ba\;(q-1+\de)\|\na u\|_q^q$.
Moreover, $M_u$ is increasing in $(0,t_{max})$ and decreasing in $(t_{max},\infty)$. As a consequence
\begin{align*}
(p-1+\de)t^{p}_{max}\|u\|^{p}&\leq (p-1+\de)t^{p}_{max}\|\na u\|_p^p+ \ba (q-1+\de)t^{q}_{max}\|\na u\|_q^q \\&=(r-1+\de) t^{r}_{max}\int_{\Om} |u|^{r} dx\leq (r-1+\de) t^{r}_{max} S^\frac{-r}{p} |\Om|^{1-\frac{r}{p^*}}\|u\|^r,
\end{align*}
set
\begin{equation*}
T_0 :=  \frac{1}{\|u\|}\left(\frac{(p-1+\de)S^\frac{r}{p}}{(r-1+\de)|\Om|^{1-\frac{r}{p^*}}}\right)^{\frac{1}{r-p}}\leq t_{max},
\end{equation*}
then,
\begin{equation*}
\begin{aligned}
M_u(t_{max})\geq  M_u(T_0)&\geq T_0^{p-1+\de} \|u\|^{p} -T_0^{r-1+\de}  S^\frac{-r}{p} |\Om|^{1-\frac{r}{p^*}}\|u\|^{r}\\&=\|u\|^{1-\de}\left(\frac{r-p}{r-1+\de}\right)\left(\frac{(p-1+\de)S^\frac{r}{p}}{(r-1+\de)|\Om|^{1-\frac{r}{p^*}}}\right)^{\frac{p-1+\de}{r-p}}\geq 0.
\end{aligned}
\end{equation*}
Therefore, if $\la < \la_*$, we have $M_u(t_{max})>\la\ds\int_\Om |u|^{1-\de} dx$, which ensures the existence of $\ul t<t_{max}<\ov t$ such that $M_u(\ul t) =M_u(\ov t)=\la \ds\int_{\Om} |u|^{1-\de} dx $. That is,  $\ul tu$ and $\ov tu \in {N}_{\la}.$ Also, $M_u^{\prime}(\ul t)>0$ and $M_u^{\prime}(\ov t)<0$ which implies $\ul tu \in{N}^{+}_{\la}$ and $\ov tu \in N^{-}_{\la}.$

\begin{Lemma}\label{lemm1}
	The following hold:
	\begin{enumerate}
		\item[(i)] $\sup\{\|u\|: u\in N_\la^+\}<\infty$
	   \item[(ii)] $\inf\{\|v\|:v\in N_\la^-\}>0$ and $\sup\{\|v\|:v\in N_\la^-, I_\la(v)\le M\}<\infty$ for all $M>0$.
	\end{enumerate}
Moreover, $\theta_{\la}^+>-\infty$ and $\theta_{\la}^->-\infty$.
\end{Lemma}
\begin{proof}
	(i) Let $u\in N_\la^+$. We have
       \begin{align*}
      	    0<J^{\prime\prime}_u(1)=(p-r)\|\na u\|_p^p+\ba (q-r)\|\na u\|_q^q+\la (r-1+\de)\int_{\Om} |u|^{1-\de}dx,
       \end{align*}
      then by means of H\"older inequality and Sobolev embeddings, we obtain
         \begin{align*}
      	        \|u\|^{p-1+\de}\le \la\frac{(r-1+\de)C}{r-p},
         \end{align*}
      which implies $\sup\{\|u\|: u\in N_\la^+\}<\infty$.\\
      (ii) Let $v\in N_\la^-$. We have
            \begin{align*}
        	      0>J^{\prime\prime}_u(1)=(p-1+\de)\|\na u\|_p^p+\ba (q-1+\de)\|\na u\|_q^q-(r-1+\de)\int_{\Om}|u|^rdx,
             \end{align*}
        which on using Sobolev embedding gives us
            \begin{align*}
        	            \frac{p-1+\de}{C(r-1+\de)}\le\|u\|^{r-p}.
             \end{align*}
       Furthermore, if $I_\la(u)\le M$, we have
           \begin{align*}
         	    I_\la(u)=\left(\frac{1}{p}-\frac{1}{r}\right)\|\na u\|_p^p +\ba\left(\frac{1}{q}-\frac{1}{r}\right)\|\na u\|_q^q -\la\left(\frac{1}{1-\de}-\frac{1}{r}\right)\int_\Om|u|^{1-\de}dx\le M,
            \end{align*}
        which implies that
            \begin{align*}
             	\left(\frac{1}{p}-\frac{1}{r}\right)\|u\|^p\le M+\la\left(\frac{1}{1-\de}-\frac{1}{r}\right)C\|u\|^{1-\de}.
            \end{align*}
      Since $1-\de<1<p$, we get the required result.\QED
\end{proof}

\begin{Lemma}\label{lemm2}
	For all $\la\in(0,\la_*)$, $\theta_{\la}^+<0$.
\end{Lemma}
\begin{proof}
	Let $u\in N_\la^+$, then using \eqref{eqb21} and \eqref{eqb22}, we have
	   \begin{align*}
	   	 I_\la(u)&=\left(\frac{1}{p}-\frac{1}{1-\de}\right)\|\na u\|_p^p +\ba\left(\frac{1}{q}-\frac{1}{1-\de}\right)\|\na u\|_q^q -\left(\frac{1}{r}-\frac{1}{1-\de}\right)\int_\Om|u|^{r}dx\\
	   	    &\le \left(\frac{1}{p}-\frac{1}{1-\de}\right)\|\na u\|_p^p +\ba\left(\frac{1}{q}-\frac{1}{1-\de}\right)\|\na u\|_q^q\\
	   	     &\qquad -\left(\frac{1}{r}-\frac{1}{1-\de}\right)\left[\frac{p-1+\de}{r-1+\de}\|\na u\|_p^p+ \frac{q-1+\de}{r-1+\de}\|\na u\|_q^q\right]\\
	   	    &=\frac{(p-1+\de)}{1-\de}\left(-\frac{1}{p}+\frac{1}{r}\right)\|\na u\|_p^p+ \ba\frac{(q-1+\de)}{r-1+\de}\left(-\frac{1}{q}+\frac{1}{r}\right)\|\na u\|_q^q<0.
	   \end{align*}
 This completes proof of the lemma.	   \QED
\end{proof}

 \begin{Lemma}\label{lemm3}
	Suppose $u\in N_\la^+$ and $v\in N_\la^-$ are minimizers of $I_\la$ on $N_\la^+$ and $N_\la^-$, respectively. Then for each $w\in W^{1,p}_0(\Om)_+$, the following hold:
	\begin{enumerate}
		\item[(i)] there exists $\e_0>0$ such that $I_\la(u+\e w)\ge I_\la(u)$ for all $0\le\e\le\e_0$,
		\item[(ii)] $t_\e\ra 1$ as $\e\ra 0^+$, where for each $\e\ge 0$, $t_\e$ is the unique positive real number satisfying $t_\e(u+\e w)\in N_\la^-$.
	\end{enumerate}
\end{Lemma}
\begin{proof}
	Let $w\in W^{1,p}_0(\Om)_+$. (i) Set
	\begin{small}
	\begin{align*}
	   \Theta(\e)=(p-1)\|\na (u+\e w)\|_p^p+\ba(q-1)\|\na (u+\e w)\|_q^q+\la\de\int_{\Om}|u+\e w|^{1-\de}dx-(r-1)\|u+\e w\|_r^r
	\end{align*}
	\end{small}
	for $\e\ge 0$. Then using continuity of $\Theta$ and the fact that $\Theta(0)=J^{\prime\prime}_u(1)>0$, there exists $\e_0>0$ such that $\Theta(\e)>0$ for all $0\le\e\le\e_0$. Since for each $\e>0$, there exists $s_\e>0$ such that $s_\e(u+\e w)\in N_\la^+$, for each $\e\in [0,\e_0]$ we have
	\begin{align*}
	I_\la(u+\e w)\ge I_\la(s_\e(u+\e w))\ge \theta_{\la}^+=I_\la(u).
	\end{align*}
	(ii) We define a $C^\infty$ function $\xi: (0,\infty)\times \mb R^4\to\mb R$ by
	\begin{align*}
	\xi(t,a,b,c,d)=at^{p-1}+b\ba\; t^{q-1}-\la ct^{-\de}-dt^{r-1}
	\end{align*}
	for $(t,a,b,c,d)\in (0,\infty)\times\mb R^4$. We have
	\begin{align*}
	&\frac{\pa \xi}{\pa t}\left(1,\|\na v\|_p^p,\|\na v\|_q^q,\int_{\Om}|v|^{1-\de},\|v\|_r^r \right)=J^{\prime\prime}_v(1)<0, \ \mbox{and} \\
	&\xi\Big(t_\e,\|\na(v+\e w)\|_p^p,\|\na(v+\e w)\|_q^q,\int_{\Om}|v+\e w|^{1-\de},\|v+\e w\|_r^r \Big)=J^\prime_{v+\e w}(t_\e)=0
	\end{align*}
	for each $\e\ge 0$. Moreover,
	\begin{align*}
	\xi\left(1,\|\na v\|_p^p,\|\na v\|_q^q,\int_{\Om}|v|^{1-\de},\|v\|_r^r \right)=J^{\prime}_v(1)=0.
	\end{align*}
	Therefore, by implicit function theorem there exist open neighbourhood $U\subset(0,\infty)$ and $V\subset\mb R^4$ containing $1$ and $\left(\|\na v\|_p^p,\|\na v\|_q^q,\int_{\Om}|v|^{1-\de},\|v\|_r^r \right)$, respectively such that for all $y\in V$, $\xi(t,y)=0$ has a unique solution $t=h(y)\in U$, where $h:V\to U$ is a continuous function.
	Since \begin{align*}
	\xi\Big(t_\e,\|\na(v+\e w)\|_p^p,\|\na(v+\e w)\|_q^q,\int_{\Om}|v+\e w|^{1-\de},\|v+\e w\|_r^r \Big)=0,
	\end{align*}
	we have
	\begin{align*}
	&\Big(\|\na(v+\e w)\|_p^p,\|\na(v+\e w)\|_q^q,\int_{\Om}|v+\e w|^{1-\de},\|v+\e w\|_r^r \Big)\in V \ \mbox{and} \\
	&h\Big(\|\na(v+\e w)\|_p^p,\|\na(v+\e w)\|_q^q,\int_{\Om}|v+\e w|^{1-\de},\|v+\e w\|_r^r \Big)=t_\e.
	\end{align*}
	Thus by continuity of $h$, we get $t_\e\ra 1$ as $\e\ra 0^+$.\QED
\end{proof}

\begin{Lemma}\label{lemm4}
	Suppose $u\in N_\la^+$ and $v\in N_\la^-$ are minimizers of $I_\la$ on $N_\la^+$ and $N_\la^-$, respectively. Then for each $w\in W^{1,p}_0(\Om)_+$, we have $u^{-\de}w,\ v^{-\de}w\in L^1(\Om)$ and
	\begin{align*}
	&\int_{\Om}\big(|\na u|^{p-2}\na u\na w+\ba|\na u|^{q-2}\na u\na w-\la u^{-\de}w-u^{r-1}w\big)dx\ge 0,\\
	&\int_{\Om}\big(|\na v|^{p-2}\na v\na w+\ba|\na v|^{q-2}\na v\na w-\la v^{-\de}w-v^{r-1}w\big)dx\ge 0.
	\end{align*}
\end{Lemma}
\begin{proof}
	Let $w\in W^{1,p}_0(\Om)_+$, then by Lemma \ref{lemm3}(i), for each $\e\in(0,\e_0)$, we have
	\begin{align*}
	 0\le\frac{I_\la(u+\e w)-I_\la(u)}{\e}=&\frac{1}{p\e}\int_\Om\big(|\na(u+\e w)|^p-|\na u|^p\big)+\ba \frac{1}{q\e}\int_\Om\big(|\na(u+\e w)|^q-|\na u|^q\big)\\
	&-\frac{\la}{(1-\de)\e}\int_{\Om}\big(|u+\e w|^{1-\de} -|u|^{1-\de}\big)dx-\frac{1}{r\e}\int_{\Om}(|u+\e w|^r-|u|^r)dx.
	\end{align*}
	It can be easily verified that as $\e\ra 0^+$
	\begin{align*}
	&\frac{1}{p\e}\int_\Om\big(|\na(u+\e w)|^p-|\na u|^p\big)dx \lra \int_{\Om}|\na u|^{p-2}\na u \na wdx \\
	&\frac{1}{q\e}\int_\Om\big(|\na(u+\e w)|^q-|\na u|^q\big)dx \lra \int_{\Om}|\na u|^{q-2}\na u \na wdx, \ \mbox{and} \\
	&\frac{1}{r\e}\int_{\Om}(|u+\e w|^r-|u|^r)dx\lra \int_\Om |u|^{r-2}uwdx,
	\end{align*}
	which imply that $\frac{|u+\e w|^{1-\de} -|u|^{1-\de}}{(1-\de)\e}\in L^1(\Om)$. For each $x\in\Om$,
	\begin{align*}
	\frac{1}{\e}\left(\frac{|u+\e w|^{1-\de}(x)-|u|^{1-\de}(x)}{1-\de}\right)
	\end{align*}
 increases monotonically as $\e\downarrow 0$ and
	\begin{align*}
	  \lim_{\e\downarrow 0}\frac{|u+\e w|^{1-\de}(x) -|u|^{1-\de}(x)}{(1-\de)\e}=\begin{cases}
	   0, \quad \mbox{if }w(x)=0,\\
	   (u(x))^{-\de}w(x) \quad \mbox{if }0<w(x), u(x)>0 \\
	   \infty \quad \mbox{if }w(x)>0, u(x)=0.
	 \end{cases}
	\end{align*}
	So, by using the monotone convergence theorem, we get $u^{-\de}w\in L^1(\Om)$ and
	\begin{align*}
	 \int_{\Om}\big(|\na u|^{p-2}\na u\na w+\ba|\na u|^{q-2}\na u\na w-\la u^{-\de}w-u^{r-1}w\big)dx\ge 0.\end{align*}
	Next, we will show these properties for $v$. For each $\e>0$, there exists $t_\e>0$ such that $t_\e(v+\e w)\in N_\la^-$. By Lemma \ref{lemm3}(ii), for sufficiently small $\e>0$, we have
	\begin{align*}
	     I_\la(t_\e(v+\e w))\ge I_\la(v)\ge I_\la(t_\e v).
	\end{align*}
	Therefore
	    $\displaystyle \frac{I_\la(v+\e w)-I_\la(t_\e v)}{\e} \ge 0,$
	which implies that
	\begin{align*}
	    \frac{\la t_\e^{1-\de}}{\e} \int_\Om (|v+\e w|^{1-\de}-|v|^{1-\de})dx\le &\frac{t_\e^p}{p\e}\int_\Om\big(|\na(v+\e w)|^p-|\na v|^p\big)dx\\
	    &\quad +\ba \frac{t_\e^q}{q\e}\int_\Om\big(|\na(v+\e w)|^q-|\na v|^q\big)dx\\
	    &\quad -\frac{t_\e^r}{r\e}\int_{\Om}(|v+\e w|^r-|v|^r)dx.
	\end{align*}
  Since $t_\e\ra 1$ as $\e\downarrow 0$, using similar arguments as in the previous case, we obtain $v^{-\de}w\in L^1(\Om)$ and
	$$\displaystyle
	\int_{\Om}\big(|\na v|^{p-2}\na v\na w+\ba|\na v|^{q-2}\na v\na w-\la v^{-\de}w-v^{r-1}w\big)dx\ge 0.  $$ \QED
\end{proof}
\begin{Theorem}\label{thmA}
	Suppose $u\in N_\la^+$ and $v\in N_\la^-$ are minimizers of $I_\la$ on $N_\la^+$ and $N_\la^-$, respectively. Then $u$ and $v$ are weak solutions of problem $(P_\la)$.
\end{Theorem}
\begin{proof}
	Let $\phi\in W^{1,p}_0(\Om)$. For $\e>0$, define $\psi\in W^{1,p}_0(\Om)$ by
	\begin{align*}
	    \psi\equiv (u+\e \phi)^+\ge 0.
	\end{align*}
  Set $\Om_+=\{x\in\Om:u(x)+\e\phi(x)\ge 0\}$, then using Lemma \ref{lemm4} and the fact $u\in N_\la$, we deduce that
	\begin{equation}
	\begin{aligned}
	0&\le \int_{\Om}\big(|\na u|^{p-2}\na u\na\psi+\ba|\na u|^{q-2}\na u\na\psi-\la u^{-\de}\psi-u^{r-1}\psi\big)dx\\
	&= \int_{\Om_+}\big(|\na u|^{p-2}\na u\na (u+\e\phi)+\ba|\na u|^{q-2}\na u\na (u+\e\phi)-\la u^{-\de}(u+\e\phi)-u^{r-1}(u+\e\phi)\big)dx\\
	&=\left(\int_{\Om}-\int_{ \{u+\e\phi\le 0\} }\right)\big(|\na u|^{p-2}\na u\na (u+\e\phi)+\ba|\na u|^{q-2}\na u\na (u+\e\phi) \\
	&\qquad\qquad\qquad\qquad\qquad-\la u^{-\de}(u+\e\phi)-u^{r-1}(u+\e\phi)\big)dx\\
	&=\int_{\Om}\big(|\na u|^p+\ba|\na u|^q-\la u^{1-\de}-u^r\big)dx\\
	&\quad+\e \int_{\Om}\big(|\na u|^{p-2}\na u\na\phi+\ba|\na u|^{q-2}\na u\na\phi-\la u^{-\de}\phi-u^{r-1}\phi\big)dx \\
	&-\int_{\{u+\e w<0\}}\big(|\na u|^{p-2}\na u\na (u+\e\phi)+\ba|\na u|^{q-2}\na u\na (u+\e\phi)-(\la u^{-\de}-u^{r-1})(u+\e\phi)\big)dx\\
	&\le \e \int_{\Om}\big(|\na u|^{p-2}\na u\na\phi+\ba|\na u|^{q-2}\na u\na\phi-\la u^{-\de}\phi-u^{r-1}\phi\big)dx\\\label{eqb33}
	&\quad-\e\int_{\{u+\e\phi<0\}}\big(|\na u|^{p-2}\na u\na \phi+\ba|\na u|^{q-2}\na u\na\phi\big)dx.
	\end{aligned}\end{equation}
	Since the measure of $\{u+\e\phi<0\}$ tends to $0$ as $\e\ra 0$, it follows that
	\begin{align*}
	\int_{\{u+\e\phi<0\}}\big(|\na u|^{p-2}\na u\na \phi+\ba|\na u|^{q-2}\na u\na\phi\big)dx\ra 0 \quad \mbox{as } k\ra 0.
	\end{align*}
	Dividing by $\e$ and letting $\e\ra 0$ in \eqref{eqb33}, we obtain
	\begin{align*}
	\int_{\Om}\big(|\na u|^{p-2}\na u\na\phi+\ba|\na u|^{q-2}\na u\na\phi-\la u^{-\de}\phi-u^{r-1}\phi\big)dx\ge 0.
	\end{align*}
	Since $\phi$ was arbitrary, this holds for $-\phi$ also. Hence, for all $\phi\in W^{1,p}_0(\Om)$, we have
	\begin{align*}
	  \int_{\Om}\big(|\na u|^{p-2}\na u\na\phi+\ba|\na u|^{q-2}\na u\na\phi-\la u^{-\de}\phi-u^{r-1}\phi\big)dx=0,
	\end{align*}
	that is $u$ is a weak solution of $(P_\la)$ and analogous arguments hold for $v$ also.\QED
\end{proof}

\section{Multiplicity results}
\subsection{subcritical case ($r<p^*$)}
In this section we prove existence and multiplicity results for weak solutions of $(P_\la)$ in the subcritical case.
\begin{Proposition}\label{prop1}
	For all $\la\in (0,\la_*)$ and $\ba>0$, there exist $u\in N_\la^+$ and $v\in N_\la^-$ such that $I_\la(u)=\theta_{\la}^+$ and $I_\la(v)=\theta_{\la}^-$.
\end{Proposition}
\begin{proof}
	Let $\{u_k\}\subset N_\la^+$ be such that $I_\la(u_k)\ra \theta_{\la}^+$ as $k\ra\infty$. By Lemma \ref{lemm1}(i), $\{u_k\}$ is bounded in $W^{1,p}_0(\Om)$, therefore without loss of generality we may assume there exists $u\in W^{1,p}_0(\Om)$ such that $u_k\rightharpoonup u$ weakly in $W^{1,p}_0(\Om)$ and $u_k(x)\ra u(x)$ a.e. in $\Om$. We claim that $u\neq 0$.
	Suppose $u=0$, then by Lemma \ref{lemm2}, we have \begin{align*}
	  0=I_\la(u)\le\ds\varliminf_{k\ra\infty} I_\la(u_k)=\theta_{\la}^+<0,\end{align*} which is a contradiction.
	 Now we will show that $u_k\ra u$ strongly in $W^{1,p}_0(\Om)$. On the contrary assume $\|\na(u_k-u)\|_p\ra a_1>0$ and $\|\na(u_k-u)\|_q\ra a_2$. By Brezis-Lieb lemma and Sobolev embeddings, we have
	  \begin{equation}\label{eqb31}
		 0=\lim_{k\ra\infty}J^\prime_{u_k}(1)=J^\prime_u(1)+a_1^p+a_2^q.
	  \end{equation}
  Since $\la\in(0,\la_*)$, by fibering map analysis there exist $0<\ul s<\ov s$ such that $J^\prime_u(\ul s)=0=J^\prime_u(\ov s)$ and $\ul su\in N_\la^+$. By \eqref{eqb31}, we get $J^\prime_u(1)<0$ which gives us $1<\ul s$ or $\ov s<1$. When $1<\ul s$, we have
    \begin{align*}
    	\theta_{\la}^+\ge J_u(1)+\frac{a_1^p}{p}+\frac{a_2^q}{q}>J_u(1)>J_u(\ul s)\ge \theta_{\la}^+,
    \end{align*}
   which is a contradiction. Thus, we have $\ov s>1$. We set $f(t)=J_u(t)+ t^p\frac{a_1^p}{p}+ t^q\frac{a_2^q}{q}$ for $t>0$. With the help of \eqref{eqb31}, we get $f^\prime(1)=0$ and $f^\prime(\ov s)=a_1^p \ov s^{p-1}+a_2^q \ov s^{q-1}>0$. So, $f$ is increasing in $[\ov s, 1]$, thus we obtain
    \begin{align*}
    	\theta_{\la}^+\ge f(1)> f(\ov s)> J_u(\ov s)> J_u(\ul s)\ge \theta_{\la}^+,
    \end{align*}
    which is also a contradiction. Hence we have $a_1=0$ that is, $u_k\ra u$ strongly in $W^{1,p}_0(\Om)$. Since $\la\in (0,\la_*)$, we get $J^{\prime\prime}_u(1)>0$, this implies that $u\in N_\la^+$ and $I_\la(u)=\theta_{\la}^+$.

  Now we will show that there exists $v\in N_\la^-$ such that $I_\la(v)=\theta_{\la}^-$. Let $\{v_k\}\subset N_\la^-$ be such that $I_\la(v_k)\ra \theta_{\la}^-$ as $k\ra\infty$. By Lemma \ref{lemm1}(ii), we may assume there exists $v\in W^{1,p}_0$ such that $v_k\rightharpoonup v$ (upto subsequence) weakly in $W^{1,p}_0(\Om)$ and $v_k(x)\ra v(x)$ a.e. in $\Om$. We will show that $v\neq 0$. If $v=0$, then $v_k$ converges to $0$ strongly in $W^{1,p}_0(\Om)$ which contradicts Lemma \ref{lemm1}(ii). We will show that $v_k\ra v$ strongly in $W^{1,p}_0(\Om)$. Suppose not, then we may assume $\|\na(v_k-v)\|_p\ra b_1>0$ and $\|\na(v_k-v)\|_q\ra b_2$. By Brezis-Lieb lemma and Sobolev embeddings, we have
  \begin{equation}\label{eqb32}
  	\theta_{\la}^-\ge I_\la(v)+\frac{b_1^p}{p}+\frac{b_2^q}{q}, \  J^\prime_v(1)+b_1^p+b_2^q=0 \quad \mbox{and } J^{\prime\prime}_v(1)+b_1^p+b_2^q\le 0.
  \end{equation}
  Since $\la\in(0,\la_*)$, $J_v^\prime(1)<0$ and $J^{\prime\prime}_v(1)<0$, there exists $\ov t\in(0,1)$ such that $\ov tv\in N_\la^-$. Set $g(t)=J_v(t)+\frac{b_1^p}{p}t^p+\frac{b_2^q}{q}t^q$ for $t>0$.  From \eqref{eqb32}, we get $g'(1)=0$ and $g'(\ov t)=b_1^p\ov t^{p-1}+b_2^q\ov t^{q-1}>0$. So, $g$ is increasing on $[\ov t,1]$ and thus we obtain
    \begin{align*}
    	\theta_{\la}^-\ge g(1)>g(\ov t)>I_\la(\ov tv)\ge \theta_{\la}^-,
    \end{align*}
   which is a contradiction. Hence, $b_1=0$ and $v_k\ra v$ strongly in $W^{1,p}_0(\Om)$. Since $\la\in(0,\la_*)$, we have $J^{\prime\prime}_v(1)<0$. Thus $v\in N_\la^-$ and $I_\la(v)=\theta_{\la}^-$.\QED
\end{proof}
 \textbf{Proof of Theorem \ref{thm1}:}  Proof follows from Proposition \ref{prop1} and Theorem \ref{thmA}.\QED

\subsection{Critical Case}
Let $\tilde{\la}_*:=\ds\sup\bigg\{\la>0: \ds\sup\{\|u\|^p:u\in N_\la^+\}\le\Big(\frac{p^*}{p}\Big)^\frac{p}{p^*-p}S^\frac{p^*}{p^*-p}\bigg\}$, then by Lemma \ref{lemm1}(i) we can see that $\tilde{\la}_*>0$. Set $\La=\ds\min\{\la_*,\tilde{\la}_*\}>0$.

\begin{Proposition}\label{prop2}
	For all $\la\in(0,\La)$ and $\ba>0$, there exists  $u_\la\in N_\la^+$ such that $\theta_{\la}^+=I_\la(u_\la)$.
\end{Proposition}
\begin{proof}
  Let $\{u_k\}\subset N_\la^+$ be such that $I_\la(u_k)\ra \theta_{\la}^+$ as $k\ra\infty$. By Lemma \ref{lemm1}(i), we get $\{u_k\}$ is bounded in $W^{1,p}_0(\Om)$, therefore we may assume there exists $u_\la\in W^{1,p}_0(\Om)$ such that $u_k\rightharpoonup u_\la$ weakly in $W^{1,p}_0(\Om)$ and $u_k(x)\ra u_\la(x)$ a.e. in $\Om$. Set $w_k=u_k-u_\la$. By Brezis-Lieb lemma, we have
  \begin{equation}\label{eqb41}
  \begin{aligned}
  	 &\theta_{\la}^++o_k(1)=I_\la(u_\la)+\frac{1}{p}\int_{\Om}|\na w_k|^p+ \ba\frac{1}{q}\int_{\Om}|\na w_k|^q-\frac{1}{p^*}\int_{\Om}|w_k|^{p^*}dx \ \ \mbox{and } \\ &\int_{\Om}\big(|\na u_\la|^p+|\na w_k|^p\big)+\ba\int_{\Om}\big(|\na u_\la|^q+|\na w_k|^q\big)=\la\int_{\Om}|u_\la|^{1-\de}dx+\int_{\Om}\big(|u_\la|^{p^*}+|w_k|^{p^*}\big)dx.
  	 \end{aligned}
  \end{equation}
  We assume
    \begin{align*}
    	\int_{\Om}|\na w_k|^p\ra l_1^p,\ \int_{\Om}|\na w_k|^q\ra l_2^q \ \ \mbox{and } \int_{\Om}|w_k|^{p^*}\ra d^{p^*}.
    \end{align*}
 We claim that $u_\la\neq 0$. If $u_\la=0$, then we have two cases:\\
  Case(a): $l_1=0$.\\
      By Lemma \ref{lemm2} and \eqref{eqb41}, we have
      \begin{align*}
      	0>\theta_{\la}^+=I_\la(0)=0,
      \end{align*}
      which is a contradiction.\\
  Case(b): $l_1\neq 0$.\\
     In this case \eqref{eqb41} implies
      \begin{align*}
      	 \theta_{\la}^+\ge\frac{1}{p}\big(l_1^p+l_2^q\big)-\frac{1}{p^*}d^{p^*}=\Big(\frac{1}{p}-\frac{1}{p^*}\Big)(l_1^p+l_2^q),
      \end{align*}
       then using the relation $Sd^p\le l_1^p$, we deduce that
         \begin{align*}
         	0>\theta_{\la}^+\ge\frac{1}{n}l_1^p\ge\frac{1}{n}S^\frac{p^*}{p^*-p}>0,
         \end{align*}
    which is also a contradiction, hence $u_\la\neq 0$. Since $\la\in(0,\La)$, there exist $0<\ul s<\ov s$ such that $J_{u_\la}^\prime(\ul s)=0=J_{u_\la}^\prime(\ov s)$ and $\ul su_\la\in N_\la^+$. We consider the following cases: \begin{enumerate}
        \item[(i)] $\ov s<1$,
        \item[(ii)] $\ov s\ge 1$ and $\frac{l_1^p}{p}-\frac{d^{p^*}}{p^*}<0$, and
        \item[(iii)] $\ov s\ge 1$ and $\frac{l_1^p}{p}-\frac{d^{p^*}}{p^*}\ge 0$.\end{enumerate}
    Case (i): Set $f(t)=  J_{u_\la}(t)+\frac{l_1^pt^p}{p}+\ba\frac{l_2^qt^q}{q}-\frac{t^{p^*}d^{p^*}}{p^*}$ for $t>0$.  Using fibering map analysis together with the fact $\ov s<1$ and \eqref{eqb41}, we have
      \begin{align*}
      	f'(1)=0 \ \mbox{and } f'(\ov s)=J_{u_\la}^\prime(\ov s)+\ov s^{p-1}l_1^p+\ba\ov s^{q-1}l_2^q-\ov s^{p^*-1}d^{p^*}\ge\ov s^{p-1}\big(l_1^p+\ba l_2^q-d^{p^*}\big)>0,
      \end{align*}
     which implies that $f$ is increasing on $[\ov s,1]$. Thus,
      \begin{align*}
      	 \theta_{\la}^+= f(1)> f(\ov {s})=J_{u_\la}(\ov {s})+\ov s^p\frac{l_1^p}{p}+\ba\ov s^q\frac{l_2^q}{q}-\ov s^{p^*}\frac{d^{p^*}}{p^*}&\ge J_{u_\la}(\ov {s})+\frac{\ov s^p}{p}\big(l_1^p+\ba l_2^q-d^{p^*}\big) \\
      	 &>J_{u_\la}(\ov s)> J_{u_\la}(\ul s)\ge \theta_{\la}^+ ,
      \end{align*}
      which is a contradiction.\\
      Case(ii): In this case we have $\frac{l_1^p}{p}-\frac{d^{p^*}}{p^*}<0$, then using  $Sd^p\le l_1^p$, and the fact $\la\in(0,\La)$, we deduce that
       \begin{align*}
       	  \sup\{ \|u\|^p  : \ u\in N_\la^+ \}\le \Big(\frac{p^*}{p}\Big)^\frac{p}{p^*-p}S^\frac{p^*}{p^*-p}< l_1^p\le \sup\{ \|u\|^p  : \ u\in N_\la^+ \},
       \end{align*}
      which is also a contradiction.\\
      Case(iii): In this case we have
        \begin{align*}
        	\theta_{\la}^+= J_{u_\la}(1)+\frac{l_1^p}{p}+\ba\frac{l_2^q}{q}-\frac{d^{p^*}}{p^*} \ge J_{u_\la}(1)\ge J_{u_\la}(\ul s)\ge \theta_{\la}^+ ,
        \end{align*}
     which implies that $\frac{l_1^p}{p}+\ba\frac{l_2^q}{q}-\frac{d^{p^*}}{p^*}=0$ and $\ul s=1$. Using \eqref{eqb41} we get $l_1=0=l_2$, hence $u_k\ra u_\la$ strongly in $W^{1,p}_0(\Om)$. Thus, $u_\la\in N_\la^+$ and $I_\la(u_\la)=\theta_\la^+$.\QED
  \end{proof}
  \textbf{Proof of Theorem \ref{thm2}:} Proof of the Theorem follows from Proposition \ref{prop2} and Theorem \ref{thmA}.\QED
  	
 Next we will show that there exists $v_\la\in N_\la^-$ such that $I_\la(v_\la)=\theta_{\la}^-$. Without loss of generality we assume $0\in\Om$. Let $\zeta\in C_c^\infty(\Om)$ such that $0\le\zeta\le 1$ in $\Om$, $\zeta(x)=1$ in $B_\mu(0)$ and $\zeta\equiv 0$ in $B_{2\mu}^c(0)$, for some $\mu>0$. Let
  \begin{align*}
  	 U_\e(x)=C_n\frac{\e^\frac{n-p}{p(p-1)}}{\big(\e^\frac{p}{p-1}+|x|^\frac{p}{p-1}\big)^\frac{n-p}{p}},
  \end{align*}
  where $\e>0$ and $C_n$ is a normalizing constant. Set $u_\e(x)=U_\e(x)\zeta(x)$ for all $x\in\Om$. Owing to regularity results we see that there exist $m,M>0$ such that $m\le u_\la(x)\le M$ for all $x\in B_{2\mu}(0)$.
  \begin{Lemma}\label{lemm8}
  	Let $\frac{2n}{n+2}<p<3$, then there exists $\ba_*>0$ such that for all  $\la\in(0,\La)$, $\ba\in(0,\ba_*)$ and sufficiently small $\e>0$,
  	 \begin{align*}
  	 	\sup\{I_\la(u_\la+tu_\e): t\ge 0\}< I_\la(u_\la)+\frac{1}{n}S^\frac{n}{p}.
  	 \end{align*}
  \end{Lemma}
\begin{proof}
  By continuity of $I_\la$ and the fact $p^*>p>q$, there exists $R_0>0$ (sufficiently large) such that
	 \begin{equation}\label{eqb50}
	    I_\la(u_\la+tu_\e)< I_\la(u_\la), \ \ \mbox{for all } t\ge R_0.
	 \end{equation}
  Next, we will show that
	   \begin{align*}
	   	 \sup_{0\le t\le R_0} I_\la(u_\la+tu_\e)< I_\la(u_\la)+\frac{1}{n}S^\frac{n}{p}.
	   \end{align*}
We have the following estimates which were proved in \cite{garcia}
    \begin{equation}\label{eqb78}
      \int_{\Om} |\na(u_\la+tu_\e)|^p\le \int_{\Om}|\na u_\la|^p+ t^p\int_{\Om}|\na u_\e|^p+ pt\int_{\Om} |\na u_\la|^{p-2} \na u_\la \na u_\e +O(\e^{\al_1}),
    \end{equation}
    with $\al_1>\frac{n-p}{p}$ and
    \begin{equation}\label{eqb79}
    	\int_\Om\big(u_\la+tu_\e\big)^{p^*}dx\ge\int_{\Om}u_\la^{p^*}+t^{p^*}\int_{\Om}u_\e^{p^*}+p^*t\int_{\Om}u_\la^{p^*-1}u_\e+ p^*t^{p^*-1}\int_\Om u_\la u_\e^{p^*-1}+O(\e^{\al_2}),
    \end{equation}
    with $\al_2>\frac{n-p}{p}$. Fix $p-1<\rho<\frac{n(p-1)}{n-p}$, then there exists $L>0$ such that
     \begin{equation}\label{eqb42}
     	\la\left(\frac{(a+b)^{1-\de}}{1-\de}-\frac{a^{1-\de}}{1-\de}-\frac{b}{a^{\de}}\right)\ge -L\; b^\rho, \ \ \mbox{for all } a\ge m \ \mbox{ and }b\ge 0.
     \end{equation}
   Let $\ba=\e^{\al_3}$, with $\al_3>\frac{n-p}{p}$. Noting the fact that $u_\la$ is a weak solution of $(P_\la)$ and taking into account \eqref{eqb78},\eqref{eqb79} and \eqref{eqb42}, we deduce that
     \begin{align*}
     	I_\la(u_\la+tu_\e)-I_\la(u_\la)=& I_\la(u_\la+tu_\e)-I_\la(u_\la)\\
     	 &-t\int_{\Om}\big(|\na u_\la|^{p-2}\na u_\la\na u_\e+\ba|\na u_\la|^{q-2}\na u_\la\na u_\e-\la u_\la^{-\de}u_\e-u_\la^{p^*-1}u_\e\big)dx\\
     	 \le &\frac{t^p}{p}\int_\Om |\na u_\e|^p+O(\e^{\al_1})+O(\e^{\al_3})+ L\;t^\rho\int_{\Om}u_\e^\rho -\frac{t^{p^*}}{p^*}\int_{\Om}u_\e^{p^*}\\ &-t^{p^*-1}\int_{\Om}u_\la u_\e^{p^*-1}+O(\e^{\al_2}).
     \end{align*}
   We have the following estimates
     \begin{equation}\label{eqb45}
     	\int_{\Om}|\na u_\e|^p =\int_{\mb R^n}|\na U_1|^p+O(\e^\frac{n-p}{p-1}), \
     	\int_{\Om} u_\e^{p^*} =\int_{\mb R^n} U_1^{p^*}+O(\e^\frac{n}{p-1}) \ \mbox{ and}
     	\int_{\Om}u_\e^\rho =O(\e^{\rho\frac{n-p}{p(p-1)}}),
     \end{equation}
    with $\rho\frac{n-p}{p(p-1)}>\frac{n-p}{p}$. Thus noting the fact that $u_\la\in L^\infty_\text{loc}(\Om)$, for $0\le t\le R_0$, we obtain
    \begin{align*}
    		I_\la(u_\la+tu_\e)-I_\la(u_\la)&\le\frac{t^p}{p}\int_{\mb R^n} |\na U_1|^p+O(\e^{\al_4}) -\frac{t^{p^*}}{p^*}\int_{\mb R^n}|U_1|^{p^*} -t^{p^*-1}C\;\e^\frac{n-p}{p}
    \end{align*}
    with $\al_4>\frac{n-p}{p}$ and $C>0$. Now following the approach of \cite{garcia}, there exists $\e_1>0$ such that
     \begin{equation*}
     	\sup_{0\le t\le R_0} I_\la(u_\la+tu_\e)< I_\la(u_\la)+\frac{1}{n}S^\frac{n}{p},
     \end{equation*}
    for all $\la\in(0,\La)$, $\e\in(0,\e_1)$ and $\ba\in(0,\ba_*)$, where $\ba_*:=\e_1^\frac{n-p}{p}$. This together with \eqref{eqb50} completes the proof.\QED
     \end{proof}

 \begin{Lemma}\label{lemm11}
 	Let $\frac{2n}{n+2}<p<3$, then for each  $\la\in(0,\La)$ and $\ba\in(0,\ba_*)$, the following holds
 	  \begin{align*}
 	      \theta_{\la}^-< I_\la(u_\la)+\frac{1}{n}S^\frac{n}{p}.
 	  \end{align*}
 \end{Lemma}
\begin{proof}
 The proof follows exactly on the same lines of \cite[Lemma 8]{hirano}.\QED
\end{proof}
 \begin{Lemma}\label{lemm7}
 	There exists a constant $D_0>0$ such that for all $u\in N_\la$,
 	  \begin{align*}
 	  	I_\la(u)\ge -D_0 \la^\frac{p}{p-1+\de}.
 	  \end{align*}
 \end{Lemma}
\begin{proof} Let $u\in N_\la$, then since $J^\prime_u(1)=0$, we have
	 \begin{equation}\label{eqb46}
	\begin{aligned}
	I_{\la}(u) &= \left(\frac{1}{p}-\frac{1}{p^*}\right) \|\na u\|_p^p +\ba \left(\frac{1}{q}- \frac{1}{p^*}\right) \|\na u\|_q^q- \la\left(\frac{1}{1-\de}-\frac{1}{p^*}\right)\int_\Om |u|^{1-\de} dx  \\
	& \geq \left(\frac{1}{p}-\frac{1}{p^*}\right)\|u\|^{p}-\la   \left(\frac{1}{1-\de}-\frac{1}{p^*}\right)\int_\Om |u|^{1-\de}dx.
	\end{aligned}
	\end{equation}
	\noi Using H\"older inequality, Sobolev embeddings and Young inequality, we deduce that
	\begin{equation}\label{eqb47}
	\begin{aligned}
	\la \int_\Om |u|^{1-\de}dx &\leq \la S^{-\frac{1-\de}{p}} |\Om|^{1-\frac{1-\de}{p^*}} \|u\|^{1-\de}\\
	&= \left(\frac{p}{1-\de}\left(\frac{1}{p}-\frac{1}{p^*}\right) \left(\frac{1}{1-\de}-\frac{1}{p^*}\right)^{-1}\right)^\frac{1-\de}{p}\|u\|^{1-\de}\\
	&\qquad \quad \la \left(\frac{p}{1-\de}\left(\frac{1}{p}-\frac{1}{p^*}\right) \left(\frac{1}{1-\de}-\frac{1}{p^*}\right)^{-1}\right)^{-\frac{1-\de}{p}} |\Om|^{1-\frac{1-\de}{p^*}} S^{-\frac{1-\de}{p}} \\
	& \leq \left(\frac{1}{p}-\frac{1}{p^*}\right) \left(\frac{1}{1-\de}- \frac{1}{p^*}\right)^{-1}\|u\|^{p}+ A \la ^{\frac{p}{p-1+\de}},
	\end{aligned}
	\end{equation}
	\noi where $A=  \left(\frac{p-1+\de}{p}\right)\left(\frac{p^*-1+\de}{p^*-p}\right)^{\frac{p(1-\de)}{p-1+\de}}  S^{-\frac{1-\de}{p-1+\de}} |\Om|^{\frac{p(p^*-1+\de)}{(p-1+\de)p^*}}.$
	Therefore, result follows from \eqref{eqb46} and \eqref{eqb47} with $D_0=   \left(\frac{1}{1-\de}-\frac{1}{p^*}\right)A.$ \QED
\end{proof}

\begin{Lemma}\label{lemm9}
	Let $p\in(1, \frac{2n}{n+2}\big]\cup[3,n)$. Then there exist  $\La_0,\;\ba_{0}>0$, and $u_0\in W^{1,p}_0(\Om)\setminus\{0\}$ such that for all $\la\in(0,\La_{0})$ and $\ba\in(0,\ba_{0})$
  $$ \ds\sup_{t\geq 0} I_\la(tu_0) <\frac{1}{n}S^\frac{n}{p}-D_0\la^\frac{p}{p-1+\de}.$$
	In particular $\theta_{\la}^{-} <\frac{1}{n}S^\frac{n}{p}-D_0 \la^\frac{p}{p-1+\de}\le\frac{1}{n}S^\frac{n}{p}+I_\la(u_\la).$
\end{Lemma}
\begin{proof} Let $\ga_0>0$ be such that for all $\la\in(0,\ga_0)$, $\frac{1}{n}S^\frac{n}{p}-D_0 \la^\frac{p}{p-1+\de}>0$ holds. Using H\"older inequality, we deduce that
	\begin{align*}
	I_\la(tu_{\e})&\leq \frac{t^{p}}{p}\|\na u_{\e}\|_p^p+ \ba \frac{t^{q}}{q}\|\na u_{\e}\|_q^q \\ &\leq
	\frac{t^{p}}{p}\|\na u_{\e}\|_p^p+C\ba \frac{t^{q}}{q}\|\na u_{\e}\|_p^p  \leq
	C(t^{p}+t^{q}).
	\end{align*}
	Therefore, there exists $t_0\ge 0$ such that \begin{equation*}
	\sup_{0\leq t\leq t_0} I_\la(tu_{\e}) < \frac{1}{n}S^\frac{n}{p}-D_0 \la^\frac{p}{p-1+\de}.\end{equation*}
  Let $h(t)= \frac{t^{p}}{p}\|\na u_{\e}\|_p^p +\ba \frac{t^{q}}{q}\|\na u_{\e}\|_q^q  -\frac{t^{p^*}}{p^*}\ds\int_{\Om}|u_{\e}|^{p^*}.$ We note that $h(0)=0$, $h(t)>0$ for $t$ small enough, $h(t)<0$ for $t$ large enough,
   and there exists $t_\e >0$ such that $\ds\sup_{t\ge 0}h(t)=h(t_\e),$ therefore
	$$0=h^\prime(t_\e)=t_\e^{p-1}\|\na u_{\e}\|_p^p +\ba t_\e^{q-1}\|\na u_{\e}\|_q^q -t_\e^{p^*-1}\ds\int_{\Om}|u_{\e}|^{p^*}$$
	which gives us
	\begin{align*}
	t_\e ^{p^*-q}&=\frac{1}{\|u_{\e}\|_{p^*}^{p^*}} \big(t_\e^{p-q}\|\na u_{\e}\|_p^p+ \ba\|\na u_{\e}\|_q^q\big)
	< C(1+t_\e^{p-q}).
	\end{align*}
	Since $p^*> p$, there exists $t_1> 0$ such that $t_\e <t_1$ for all  $\e >0$.
  Thus,  we  get
	\begin{equation}\label{eqb51}
	 \begin{aligned}
	    \ds\sup_{t\ge t_0} I_\la(tu_{\e})&\le \ds\sup_{t>0}h(t)-\frac{t_0^{1-\de}}{1-\de}\la\ds\int_{B_\mu(0)}|U_\e|^{1-\de}\\
	    &\leq \ds\sup_{t\ge 0}\left(\frac{t^{p}}{p}\|\na u_{\e}\|_p^p- \frac{t^{p^*}}{p^*}\|u_{\e}\|_{p^*}^{p^*}\right)
	    +\ba \frac{t_1^{q}}{q}\|\na u_{\e}\|_q^q  -\frac{t_0^{1-\de}}{1-\de}\la\ds\int_{B_\mu(0)}|U_\e|^{1-\de}.
	  \end{aligned}
	\end{equation}
  Set $g(t)=\frac{t^{p}}{p}\|u_{\e}\|^p- \frac{t^{p^*}}{p^*}\|u_{\e}\|_{p^*}^{p^*}$. A simple computation shows that $g$ attains maximum at $\tilde{t}=\left(\frac{\|u_{\e}\|^p}{\|u_{\e}\|_{p^*}^{p^*}}\right)^\frac{1}{p^*-p}$
	and \begin{align*}
	      \ds\sup_{t\ge 0}g(t)= g(\tilde{t})= \frac{1}{n}	\left(\frac{\|u_{\e}\|^{p}}{\|u_{\e}\|_{p^*}^{p^*}}\right)^\frac{n}{p},
	    \end{align*}
which on using \eqref{eqb45} reduces to
  \begin{equation*}
  	 \ds\sup_{t\ge 0}g(t)\le\frac{1}{n}S^\frac{n}{p}+C_3\;\e^\frac{n-p}{p-1}.
  \end{equation*}
  Let $\ba=\e^{\al_9}$, with $\al_9>\frac{n-p}{p-1}$. For $n-\frac{n-p}{p-1}(1-\de)>0$, we have
    \begin{align*}
       \int_{B_\mu(0)}|U_\e|^{1-\de}dx&=\e^{n-(1-\de)\frac{n-p}{p}}\int_{B_{\mu/\e}(0)} \frac{1}{\big(1+|y|^\frac{p}{p-1}\big)^{\frac{n-p}{p}(1-\de)}}dy \\
    	&\ge \e^{n-(1-\de)\frac{n-p}{p}} \int_{1}^{\mu/\e}\frac{r^{n-1}}{(1+r^\frac{p}{p-1})^{\frac{n-p}{p}(1-\de)}}dr\ge C\;\e^{\frac{n-p}{p(p-1)}(1-\de)}.
    \end{align*}
   Furthermore, for $n-\frac{n-p}{p-1}(1-\de)\le 0$, following the approach of \cite[Lemma 1.46]{willem}, we have
    \begin{align*}
    	\int_{B_\mu(0)}|U_\e|^{1-\de}dx \ge C\begin{cases}
    	         \e^{n-\frac{n-p}{p}(1-\de)}, \ \mbox{ if } n-\frac{n-p}{p-1}(1-\de)< 0,\\
    	         \e^\frac{n}{p} |\ln\e|, \ \mbox{ if } n-\frac{n-p}{p-1}(1-\de)= 0.
    	\end{cases}
    \end{align*}
   Now collecting all the informations done so far in \eqref{eqb51}, we deduce that
    \begin{equation}\label{eqb52}
    \begin{aligned}
    	 \ds\sup_{t\ge t_0} I_\la(tu_{\e})&\le \frac{1}{n}S^\frac{n}{p}+C_3\;\e^\frac{n-p}{p-1} -C_4\la\begin{cases}
    	    \e^{\frac{n-p}{p(p-1)}(1-\de)}, \mbox{ if }\de>\frac{2n-np-p}{n-p}, \\
    	    \e^{n-\frac{n-p}{p}(1-\de)}, \mbox{ if }\de<\frac{2n-np-p}{n-p}, \\
    	    \e^\frac{n}{p} |\ln\e|, \ \mbox{ if }\de=\frac{2n-np-p}{n-p}.
    	 \end{cases}
    	 \end{aligned}
    \end{equation}
    We consider the following cases:\\
   \textbf{Case(1)}: If $\frac{2n-np-p}{n-p}<\de<1$.\\
   In this case since $\frac{n-p}{p(p-1)}(1-\de)<\frac{n-p}{p-1}$, there exists $\hat\e>0$ and $\hat\ga>0$ such that for all $\e\in(0, \hat\e)$ and $\la\in(0, \hat\ga)$, we have
      \begin{align*}
      	 \sup_{t\ge t_0} I_\la(tu_\e)< \frac{1}{n}S^\frac{n}{p}-D_0\la^\frac{p}{p-1+\de},
      \end{align*}
   for all $\ba\in(0, \hat\ba)$, where $\hat\ba:=\hat\e^\frac{n-p}{p-1}$.\\ 
  \textbf{Case(2)}: If $0<\de\le\frac{2n-np-p}{n-p}$.\\
   Let $\e =\big(\la^\frac{p}{p-1+\de}\big)^\frac{p-1}{n-p}\le\mu$. Then,  \eqref{eqb52} reduces to
       \begin{equation}\label{eqb53}
        \begin{aligned}
          \ds\sup_{t\ge t_0} I_\la(tu_{\e})&\le \frac{1}{n}S^\frac{n}{p}+C_3\;\la^\frac{p}{p-1+\de} -C_4\la\begin{cases}
           \la^{\frac{p}{p-1+\de} \frac{p-1}{n-p}\big(n-\frac{n-p}{p}(1-\de)\big)}, \mbox{ if }\de<\frac{2n-np-p}{n-p}, \\
           \la^\frac{1-\de}{p-1+\de} |\ln\la^\frac{p(1-\de)}{n(p-1+\de)}|, \ \mbox{ if } \de=\frac{2n-np-p}{n-p}.
        \end{cases}
       \end{aligned}
       \end{equation}
   Subcase (2)(a): If $\de<\frac{2n-np-p}{n-p}$. \\
      In this case we have $n<\frac{n-p}{p-1}(1-\de)$, which implies that
        \begin{align*}
        	1+\frac{p}{p-1+\de} \frac{p-1}{n-p}\Big(n-\frac{n-p}{p}(1-\de)\Big)<\frac{p}{p-1+\de}.
        \end{align*}
     Therefore there exists $\ga_2>0$ such that for all $\la\in(0, \ga_2)$, we have
      \begin{align*}
      	 C_3\;\la^\frac{p}{p-1+\de}- C_4\la\cdot\la^{\frac{p}{p-1+\de} \frac{p-1}{n-p}\big(n-\frac{n-p}{p}(1-\de)\big)} < -D_0 \la^\frac{p}{p-1+\de}.
      \end{align*}
    Subcase(2)(b): If $\de=\frac{2n-np-p}{n-p}$. \\
     Since $|\ln\la^\frac{p(1-\de)}{n(p-1+\de)}|\ra\infty$ as $\la\ra 0$, there exists $\ga_3>0$ such that
       \begin{align*}
       	  C_3\;\la^\frac{p}{p-1+\de}- C_4\la\cdot\la^\frac{1-\de}{p-1+\de} |\ln\la^\frac{p(1-\de)}{n(p-1+\de)}|< -D_0 \la^\frac{p}{p-1+\de},
       \end{align*}
     for all $\la\in(0, \ga_3)$. Let $\La_{0}=\min\{ \mu^{\frac{n-p}{p-1}}, \ga_0,\la_*,\ga_2, \ga_3,\hat\ga\}>0$ and $\ba_{1} =\La_{0}^\frac{p}{p-1+\de}$. Then from \eqref{eqb53}, for all $\la\in(0,\La_{0})$ and $\ba\in(0,\ba_{1})$, we have
        \begin{align*}
        	\sup_{t\ge t_0}I_\la(tu_\e)<\frac{1}{n}S^\frac{n}{p}-D_0\la^\frac{p}{p-1+\de},
        \end{align*}
     for sufficiently small $\e>0$. Thus, for all $\la\in(0,\La_{0})$ and $\ba\in(0,\ba_{0})$, we get
      \begin{align*}
     \sup_{t\ge 0}I_\la(tu_\e)<\frac{1}{n}S^\frac{n}{p}-D_0\la^\frac{p}{p-1+\de},
     \end{align*}
     where $\ba_0=\min\{\ba_1,\hat\ba\}>0$, which proves the first part of the lemma. For the last part we observe that $u_\e\in W^{1,p}_0(\Om)\setminus\{0\}$ and since $0<\la<\la_*$, there exists $\bar{t}>0$ such that $\bar{t}u_\e\in N_\la^-$.
        Hence, \begin{align*}
             \theta_\la^- \leq I_\la(\bar{t}u_\e)\leq \ds\sup_{t\ge 0}  I_\la(tu_\e)<\frac{1}{n}S^\frac{n}{p}-D_0\la^\frac{p}{p-1+\de},
        \end{align*}
        this together with lemma \ref{lemm7} completes the proof for $u_0=u_\e$.  \QED
\end{proof}
 \begin{Proposition}\label{prop3}
 	There exists $v_\la\in N_\la^-$ such that $I_\la(v_\la)=\theta_{\la}^-$ in each of the following cases:  \begin{enumerate}
 	\item[(i)] for all $\la\in(0,\La)$ and $\ba\in(0,\ba_*)$, when $\frac{2n}{n+2}<p<3$,
 	\item[(ii)] for all $\la\in(0,\La_{0})$ and $\ba\in(0, \ba_{0})$, when $p\in(1,\frac{2n}{n+2}\big]\cup[3,n)$. \end{enumerate}
 \end{Proposition}
\begin{proof}
	Let $\{v_k\}\subset N_\la^-$ be such that $I_\la(v_k)\ra\theta_\la^-$ as $k\ra\infty$. By Lemma \ref{lemm1}(ii), we may assume there exists $v_\la\in W^{1,p}_0$ such that $v_k\rightharpoonup v_\la$ weakly in $W^{1,p}_0(\Om)$ and $v_k(x)\ra v_\la(x)$ a.e. in $\Om$ (upto subsequence). Set $z_k=v_k-v_\la$, then by Brezis-Lieb lemma, we have
	\begin{equation}\label{eqb54}
	\begin{aligned}
	    &\theta_{\la}^-+o_k(1)=I_\la(v_\la)+\frac{1}{p}\int_{\Om}|\na z_k|^p+ \ba\frac{1}{q}\int_{\Om}|\na z_k|^q-\frac{1}{p^*}\int_{\Om}|z_k|^{p^*}dx \ \ \mbox{and} \\
	    &\int_{\Om}\big(|\na v_\la|^p+|\na z_k|^p\big)+\ba\int_{\Om}\big(|\na v_\la|^q+|\na z_k|^q\big) =\la\int_{\Om}|v_\la|^{1-\de}dx+\int_{\Om}\big(|v_\la|^{p^*}+|z_k|^{p^*}\big)dx.
	  \end{aligned}
	\end{equation}
  We assume
	\begin{align*}
	   \int_{\Om}|\na z_k|^p\ra l_1^p,\ \int_{\Om}|\na z_k|^q\ra l_2^q \ \ \mbox{and } \int_{\Om}|z_k|^{p^*}\ra d^{p^*}.
	\end{align*}
	We claim that $v_\la\neq 0$. On the contrary suppose $v_\la=0$, then  by Lemma \ref{lemm1}(ii), $l_1\neq 0$. Using the relation $Sd^p\le l_1^p$ and \eqref{eqb54}, we deduce that
	\begin{align*}
	   \theta_{\la}^-=I_\la(0)+\frac{1}{p}l_1^p+\frac{\ba}{q}l_2^q-\frac{1}{p^*}d^{p^*}\ge\Big(\frac{1}{p}-\frac{1}{p^*}\Big)(l_1^p+\ba l_2^q)\ge\frac{1}{n} S^\frac{n}{p}.
	\end{align*}
	Now we consider the following cases:\\
Case(i): If $\frac{2n}{n+2}<p<3$, then by Lemma \ref{lemm11}, we have
    \begin{align*}
    	\theta_{\la}^++\frac{1}{n}S^\frac{n}{p}=I_\la(u_\la)+\frac{1}{n} S^\frac{n}{p}>\theta_\la^-\ge \frac{1}{n} S^\frac{n}{p},
    \end{align*}
  this implies $\theta_{\la}^+>0$, which contradicts lemma \ref{lemm2}. \\
 Case(ii): If $p\in(1,\frac{2n}{n+2})\cup[3,n)$, then by lemma \ref{lemm9}, we have
    \begin{align*}
    	\frac{1}{n}S^\frac{n}{p}-D_0\la^\frac{p}{p-1+\de}>\theta_{\la}^-\ge\frac{1}{n} S^\frac{n}{p},
    \end{align*}
    which is also a contradiction. Hence in all cases we get $v_\la\neq 0$. From the assumption $0<\la<\la_*$, there exist $0<\ul t<\ov t$ such that $J^\prime_{v_\la}(\ul t)=0=J^\prime_{v_\la}(\ov t)$, and $\ul tv_\la\in N_\la^-$ and $\ov tv_\la\in N_\la^-$. We define $\eta,\; f:(0,\infty)\to\mb R$ as
      \begin{align*}
      	 \eta(t)=\frac{l_1^p}{p}t^p+\ba\frac{l_2^q}{q}t^q-\frac{t^{p^*}d^{p^*}}{p^*} \ \mbox{and } f(t)=J_{v_\la}(t)+\eta(t) \ \mbox{for }t>0.
      \end{align*}
     We consider the following cases:\begin{enumerate}
       \item[(a)]$\ov t<1$,
      \item[(b)] $\ov t\ge 1$ and $d>0$, and
      \item[(c)] $\ov t\ge 1$ and $d=0$.\end{enumerate}
      Case (a): Using \eqref{eqb54}, we get $f^\prime(1)=0$, and $f^\prime(\ov t)=\ov t^{p-1}l_1^p+\ba \ov t^{q-1}l_2^q -\ov t^{p^*-1}d^{p^*}\ge\ov t^{p-1}\big(l_1^p+\ba l_2^q-d^{p^*}\big)>0$. Therefore we see that $f$ is increasing on $[\ov t,1]$. Thus
      \begin{align*}
      	 \theta_{\la}^-= f(1)> f(\ov {t})\ge J_{v_\la}(\ov {t})+\frac{\ov t^p}{p} \big(l_1^p+\ba l_2^q- d^{p^*}\big)
      	 >J_{v_\la}(\ov t)> J_{v_\la}(\ul t)\ge \theta_{\la}^- ,
      	\end{align*}
      	which is a contradiction.\\
      Case (b): It is easy to see that there exists $t_m>0$ such that $\eta(t_m)\ge\frac{1}{n}S^\frac{n}{p}$, $\eta^\prime(t_m)=0$, $\eta^\prime(t)>0$ for all $0<t<t_m$ and $\eta^\prime(t)<0$ for all $t>t_m$. By the assumption $0<\la<\la_*$, we have $f(1)=\ds\max_{t\ge 0} f(t)\ge f(t_m)$. So, if $t_m\le 1$
        \begin{equation}\label{eqb55}
        	\theta_{\la}^-= f(1)\ge f(t_m)= J_{v_\la}(t_m)+\eta(t_m)\ge I_\la(\ul tv_\la)+\frac{1}{n}S^\frac{n}{p}\ge I_\la(u_\la)+\frac{1}{n}S^\frac{n}{p},
        \end{equation}
        which is a contradiction to lemma \ref{lemm11} and lemma \ref{lemm9}. Thus we have $t_m>1$. Since $f^\prime(t)\le 0$ for all $t\in[1,t_m]$, we have $J^\prime_{v_\la}(t)\le -\eta^\prime(t)\le 0$ for all $t\in [1,t_m]$. This gives either $t_m\le \ul t$ or $\ov t=1$. If $t_m\le\ul t$, then \eqref{eqb55} holds which yields a contradiction. Hence, $\ov t=1$, that is $v_\la\in N_\la^-$ and we have
      \begin{align*}
      	 \theta_{\la}^-= f(1)= I_\la(v_\la)+\frac{l_1^p}{p}+\ba\frac{l_2^q}{q}-\frac{d^{p^*}}{p^*}\ge I_\la(v_\la)+\frac{1}{n}S^\frac{n}{p}\ge I_\la(\ul tv_\la)+\frac{1}{n}S^\frac{n}{p}\ge I_\la(u_\la)+\frac{1}{n}S^\frac{n}{p},
      \end{align*}
     which is also a contradiction.\\
     Consequently only (c) holds. If $l_1\neq 0$, then we have $J^\prime_{v_\la}(1)<0$ and $J^{\prime\prime}_{v_\la}(1)<0$ which contradicts the fact that $\ov t\ge 1$.
     Thus $l_1=0$ that is, $v_k\ra v_\la$ strongly in $W^{1,p}_0(\Om)$. Therefore, $v_\la\in N_\la^-$ and $I_\la(v_\la)=\theta_{\la}^-$. \QED
\end{proof}
\textbf{Proof of Theorem \ref{thm3}:} Proof of the Theorem follows from Propositions \ref{prop2}, \ref{prop3} and Theorem \ref{thmA}.\QED

{
	Now we will prove the existence of second solution for all $\ba>0$ in the case $r=p^*$.
\begin{Lemma}\label{lemm19}
 Let $p\in\big(\frac{2n}{n+2},3)$, then	for all $\ba>0$ and $\la\in(0,\La)$  following holds
	$$ \ds\sup_{t\ge 0}I_\la(u_\la+tu_\e)<\frac{1}{n}S^\frac{n}{p}+I_\la(u_\la)$$
	in each of the following cases:
	\begin{enumerate}
		\item[(1)]  $\max\{p-1,1\}<q<\frac{n(p-1)}{n-1}$,
		\item[(2)]  $\frac{n(p-1)}{n-1}<q<\frac{n(p-1)+p}{n}$.
     \end{enumerate}
\end{Lemma}
\begin{proof}
	Using the following one dimensional inequality
	\begin{align*}
	(1+t^2+2t\cos\al)^\frac{q}{2}\le \begin{cases}
	1+t^q+qt\cos\al+Ct^\nu, \ \mbox{ if } 1<q<2, \mbox{ for all }\nu\in(1,q),\\
	1+t^q+qt\cos\al+Ct^\nu, \ \mbox{ if } 2\le q<3, \mbox{ for all }\nu\in[q-1,2], \end{cases}
	\end{align*}
	we can prove
	\begin{equation}\label{eqb81}
	  \int_\Om |\na(u_\la+tu_\e)|^q \le
	  \int_\Om\big(|\na u_\la|^q+t^q|\na u_\e|^q+qt|\na u_\la|^{q-2}\na u_\la \na u_\e+C|\na u_\la|^{q-\nu}|\na u_\e|^\nu \big),
	\end{equation}
	for all $\nu\in(1,q)$ if $1<q<2$ and $\nu\in[q-1,2]$ if $2\le q<3$.  Moreover, we have
	\begin{align}\label{eqb84}
	\int_{\Om}|\na u_\e|^l \le C\begin{cases}
	\e^{\frac{n-p}{p(p-1)}l}, \quad \mbox{if } 1\le l<\frac{n(p-1)}{n-1} \\
	\e^{n-\frac{n}{p}l}, \quad \ \mbox{ if } \frac{n(p-1)}{n-1}<l<p.
	\end{cases}
	\end{align}
	From \eqref{eqb50} it follows that we need to prove
	  $$\sup_{0\le t\le R_0}I_\la(u_\la+tu_\e)<\frac{1}{n}S^\frac{n}{p}+I_\la(u_\la).$$
	Using \eqref{eqb78}, \eqref{eqb79}, \eqref{eqb42} and \eqref{eqb81}, we deduce that
	\begin{align}\label{eqb83}
	I_\la(u_\la+tu_\e)-I_\la(u_\la)&= I_\la(u_\la+tu_\e)-I_\la(u_\la)\nonumber\\
	&-t\int_{\Om}\big(|\na u_\la|^{p-2}\na u_\la\na u_\e+\ba|\na u_\la|^{q-2}\na u_\la\na u_\e-\la u_\la^{-\de}u_\e-u_\la^{p^*-1}u_\e\big)dx\nonumber\\
	&\le \frac{t^p}{p}\int_\Om |\na u_\e|^p+\ba\frac{t^q}{q}\int_\Om |\na u_\e|^q +C\int_{\Om}|\na u_\la|^{q-\nu} |\na u_\e|^\nu+ L\;t^\rho\int_{\Om}u_\e^\rho \nonumber\\ &\quad-\frac{t^{p^*}}{p^*}\int_{\Om}u_\e^{p^*}-t^{p^*-1}\int_{\Om}u_\la u_\e^{p^*-1}+O(\e^{\al_4}),
	\end{align}
	where $\al_4>(n-p)/p$.
	Now we consider following cases:\\
	\noi\textbf{Case (1)}: If $\max\{p-1,1\}<q<\frac{n(p-1)}{n-1}$.\\
	Since $\max\{p-1,1\}<q<\frac{n(p-1)}{n-1}$, we choose $\nu>1$ such that
	\begin{align*}
	  \nu\in\begin{cases}
		(1,q)\cap(p-1,\frac{n(p-1)}{n-1}\big) \qquad \mbox{if }1<q<2,\\
		[q-1,2]\cap(p-1,\frac{n(p-1)}{n-1}\big) \ \mbox{ if }2\le q<3.
		\end{cases}
	\end{align*}
	Then, using the fact $|\na u_\la|\in L^\infty_\text{loc}(\Om)$ and \eqref{eqb84}, we obtain
	\begin{align*}
	\int_{\Om}|\na u_\e|^q \le C \e^{l_1} \ \mbox{ and } \int_{\Om}|\na u_\la|^{q-\nu} |\na u_\e|^\nu\le C\e^{l_2},
	\end{align*}
	where $l_1, l_2>\frac{n-p}{p}$. Thus, for $0\le t\le R_0$, taking into account \eqref{eqb45} and \eqref{eqb83}, we deduce that
	\begin{align*}
	I_\la(u_\la+tu_\e)-I_\la(u_\la)&\le \frac{t^p}{p}\int_\Om |\na U_1|^p  -\frac{t^{p^*}}{p^*} \int_{\Om}|U_1|^{p^*} -t^{p^*-1}C\e^\frac{n-p}{p}+O(\e^{l_3}),
	\end{align*}
	where $l_3>\frac{(n-p)}{p}$. Following the approach as in Lemma \ref{lemm8} we get the required result in this case.
	
	\noi\textbf{Case(2)}: If $\frac{n(p-1)}{n-1}<q<\frac{n(p-1)+p}{n}$.\\
	 We note that there exists
	 \begin{align*}
	 	\nu\in\begin{cases}
	 	  (1,q)\cap\big(\frac{n(p-1)}{n-1},\infty) \qquad \mbox{if }1<q<2,\\
	 	   [q-1,2]\cap\big(\frac{n(p-1)}{n-1},\infty) \ \mbox{ if }2\le q<3.
	 	\end{cases}
	 \end{align*}
	  In this case using \eqref{eqb84} and \eqref{eqb45} in \eqref{eqb83}, we deduce that
	 \begin{align*}
	   I_\la(u_\la+tu_\e)-I_\la(u_\la)&\le \frac{t^p}{p}\int_\Om |\na U_1|^p + C_1\e^{n-\frac{n}{p}q}+C_2\e^{n-\frac{n}{p}\nu} -\frac{t^{p^*}}{p^*} \int_{\Om}|U_1|^{p^*}-t^{p^*-1}C\e^\frac{n-p}{p}\\ &\qquad+O(\e^{l_4}),
	 \end{align*}
	where $l_4>(n-p)/p$. Using the fact that $\nu<q<\frac{n(p-1)+p}{n}$, we have $n-\frac{n}{p}\nu>n-\frac{n}{p}q>\frac{n-p}{p}$, and hence
	 \begin{align*}
	 I_\la(u_\la+tu_\e)-I_\la(u_\la)&\le \frac{t^p}{p}\int_\Om |\na U_1|^p  -\frac{t^{p^*}}{p^*} \int_{\Om}|U_1|^{p^*} -t^{p^*-1}C\e^\frac{n-p}{p}+O(\e^{l_6}),
	 \end{align*}
	 where $l_6>(n-p)/p$. Now approaching as Case(1) we can complete the proof.\QED
	\end{proof}

\textbf{Proof of Theorem \ref{thm6}}: With the help of Lemma \ref{lemm19} approaching the proof in same way as in Lemma \ref{lemm11} we can show that $\theta_{\la}^-<I_\la(u_\la)+\frac{1}{n}S^\frac{n}{p}$ for all $\la\in(0,\La)$ and $\ba>0$. Then following the proof of Proposition \ref{prop3} we get $v_\la\in N_\la^-$ such that $I_\la(v_\la)=\theta_{\la}^-$ for all $\la\in(0,\La)$ and $\ba>0$. Now with the help of Theorem \ref{thmA} we see that $v_\la$ is a solution of $(P_\la)$.\QED
}

\section{Global existence result}
In this section we prove the global existence and non existence result (for all $\la$ and $\ba$) for problem $(P_\la)$.  Let us define \[\La^*=\ds\sup\{\la>0: (P_\la) \text{ has a solution} \}.\]
\begin{Lemma}\label{lemm16}
	We have $0<\La^*<\infty$.
\end{Lemma}
\begin{proof}
  With the help of Theorems \ref{thm1} and \ref{thm2}, we infer that $\La^*\ge\la_*\ge\La>0$. Next, we will show $\La^*<\infty$. On the contrary suppose there exists a non-decreasing sequence $\{\la_k\}$ such that $\la_k\ra\infty$ as $k\ra\infty$ and $(P_{\la_k})$ has a solution $u_k$.  There exists $\ul \la>0$ such that
	  \begin{align*}
	  	\frac{\la}{t^\de}+t^{r-1}\ge(\la_1(q,\ba)-\e)t^{q-1}, \ \text{for all }t>0, \e>0 \ \text{and } \la>\ul \la.
	  \end{align*}
	Choose $\la_m>\ul\la$, then $u_m$ is a super solution of
	  \begin{equation*}
	  (Q_\e)\left\{\begin{array}{rllll}
	  -\Delta_{p}u-\ba\Delta_{q}u & = (\la_1(q,\ba)-\e) u^{q-1}, \ u>0 \;  \text{ in } \Om \\ u&=0 \quad \text{ on } \pa\Om,
	  \end{array}
	  \right.
	  \end{equation*}
	  that is, for all $\phi\in W^{1,p}_0(\Om)$ with $\phi\ge 0$, we have
	   \begin{align*}
	   	  \int_\Om(|\na u_m|^{p-2}\na u_m+\ba|\na u_m|^{q-2}\na u_m)\na\phi ~dx\ge \int_\Om \big( (\la_1(q,\ba)-\e)u_m^{q-1}\big)\phi~ dx.
	   \end{align*}
	   We choose $\varrho>0$ small enough such that $\varrho\hat{\phi}<u_m$ (this can be done because of Theorem \ref{thm4}) and $\varrho\hat{\phi}$ is a subsolution of $(Q_\e)$. That is, for all $\phi\in W^{1,p}_0(\Om)$ with $\phi\ge 0$, we have
	   \begin{align*}
	   	  \int_\Om\big(|\na (\varrho\hat{\phi})|^{p-2}\na (\varrho\hat{\phi})+\ba|\na (\varrho\hat{\phi})|^{q-2}\na(\varrho\hat{\phi}) \big)\na\phi ~dx\le  \big( (\la_1(q,\ba)-\e)\int_\Om(\varrho\hat{\phi})^{q-1}\big)\phi~ dx.
	   \end{align*}
	  By monotone iteration procedure we obtain a solution $w$ for $(Q_\e)$ for $\e>0$ such that $0<\varrho\hat{\phi}\le w\le u_m$, which contradicts \cite[Theorem 1]{tanaka}. This completes the proof of Lemma.\QED
\end{proof}

\begin{Lemma}\label{lemm12}
	Let $\ul u, \ov u\in W^{1,p}_0(\Om)$ be such that $\ul u$ is a weak subsolution  and $\ov u$ is a weak supersolution of $(P_\la)$ satisfying $\ul u\le\ov u$ a.e. in $\Om$. Then there exists a weak solution $u\in W^{1,p}_0(\Om)$ of $(P_\la)$ such that $\ul u\le u\le \ov u$ a.e. in $\Om$.
\end{Lemma}
\begin{proof}
	The proof given here is an adaptation of \cite{haitao}. Set
	 $M:= \{ u\in W^{1,p}_0(\Om): \ul u\le u\le \ov u \text{ a.e. in }\Om\},$
	 then $M$ is closed and convex. It is easy to verify that $I_\la$ is weakly lower semicontinuous on $M$. Therefore, there exists a relative minimizer $u$ of $I_\la$ on $M$. We will show that $u$ is a weak solution of $(P_\la)$. For $\phi\in W^{1,p}_0(\Om)$ and $\e>0$, let $v_\e=u+\e\phi-\phi^\e+\phi_\e\in M$, where
	   \begin{align*}
	       \phi^\e:=(u+\e\phi-\ov u)_+\ge 0 \ \mbox{and } \phi_\e:=(u+\e\phi-\ul u)_-\ge 0.
	   \end{align*}
	For $0<t<1$ we see that $u+t(v_\e-u)\in M$. Therefore using the fact that $u$ is a relative minimizer of $I_\la$ on $M$, we have
	  \begin{align*}
	 	0&\le \lim_{t\ra 0}\frac{I_\la(u+t(v_\e-u))-I_\la(u)}{t}\\
	 	&= \int_\Om \big(|\na u|^{p-2}\na u+\ba|\na u|^{q-2}\na u\big)\na(v_\e-u)-\int_\Om (\la u^{-\de}+u^{r-1})(v_\e-u)dx,
	 \end{align*}
	which on using definition of $v_\e$ simplifies to
	 \begin{equation}\label{eqb71}
	 	\int_{\Om}\big(|\na u|^{p-2}\na u\na\phi+\ba|\na u|^{q-2}\na u\na\phi-\la u^{-\de}\phi-u^{r-1}\phi\big)dx\ge \frac{1}{\e}\big(E^\e-E_\e\big),
	 \end{equation}
	where
	 \begin{align*}
	 	&E^\e =\int_{\Om}\big(|\na u|^{p-2}\na u\na\phi^\e+\ba|\na u|^{q-2}\na u\na\phi^\e-\la u^{-\de}\phi^\e-u^{r-1}\phi^\e\big)dx \ \mbox{ and}\\
	 	&E_\e=\int_{\Om}\big(|\na u|^{p-2}\na u\na\phi_\e+\ba|\na u|^{q-2}\na u\na\phi_\e-\la u^{-\de}\phi_\e-u^{r-1}\phi_\e\big)dx.
	 \end{align*}
  Now we will estimate $\frac{1}{\e}E^\e$. For this, set $\Om^\e=\{x\in\Om: (u+\e\phi)(x)\ge\ov u(x)>u(x)\}$. Then
   \begin{align*}
   	 \int_{\Om}|\na u|^{p-2}\na u\na\phi^\e&=\int_{\Om^\e}|\na u|^{p-2}\na u\na(u+\e\phi-\ov u) \\
   	 &=\int_{\Om^\e}\big(|\na u|^{p-2}\na u-|\na\ov u|^{p-2}\na\ov u\big)\na(u-\ov u)  +\int_{\Om^\e}|\na\ov u|^{p-2}\na\ov u\;\na(u-\ov u) \\
   	 &\qquad +\e \int_{\Om^\e}|\na u|^{p-2}\na u\na\phi \\
   	 &\ge C_p\begin{cases}
   	    \ds\int_{\Om^\e}|\na(u-\ov u)|^p, \ \mbox{ if }p\ge 2,\\
   	    \ds\int_{\Om^\e}\frac{|\na(u-\ov u)|^2}{\big(|\na u|+|\na\ov u|\big)^{2-p}}, \ \mbox{ if } 1<p<2,
   	 \end{cases}
   	 +\int_{\Om^\e}|\na\ov u|^{p-2}\na\ov u\;\na(u-\ov u)\\
   	 & \qquad +\e \int_{\Om^\e}|\na u|^{p-2}\na u\na\phi \\
   	 &\ge \int_{\Om^\e}|\na\ov u|^{p-2}\na\ov u\;\na(u-\ov u)+\e \int_{\Om^\e}|\na u|^{p-2}\na u\na\phi.
   \end{align*}
   Similar result holds for $\ds \int_{\Om}|\na u|^{q-2}\na u\na\phi^\e$ also. Thus we obtain
    \begin{align*}
    	E^\e&\ge \int_{\Om^\e}\big(|\na\ov u|^{p-2}\na\ov u+\ba|\na\ov u|^{q-2}\na\ov u\big)\na(u-\ov u)+\e \int_{\Om^\e}\big(|\na u|^{p-2}\na u+\ba|\na u|^{q-2}\na u\big)\na\phi \\
    	&\qquad-\int_{\Om^\e}(\la u^{-\de}+u^{r-1})\phi^\e \\
    	&\ge\int_{\Om}\big(|\na\ov u|^{p-2}\na\ov u+\ba|\na\ov u|^{q-2}\na\ov u\big)\na\phi^\e-\int_{\Om^\e}(\la u^{-\de}+u^{r-1})\phi^\e \\
    	&\qquad+\e\int_{\Om^\e}\big(|\na u|^{p-2}\na u+\ba|\na u|^{q-2}\na u-|\na\ov u|^{p-2}\na\ov u-\ba|\na\ov u|^{q-2}\na\ov u\big)\na\phi,
    \end{align*}
    which on using the fact that $\ov u$ is a weak super solution of $(P_\la)$, implies
     \begin{align*}
     	E^\e&\ge \e\int_{\Om^\e}\big(|\na u|^{p-2}\na u-\ba|\na u|^{q-2}\na u-|\na\ov u|^{p-2}\na\ov u+\ba|\na\ov u|^{q-2}\na\ov u\big)\na\phi \\
     	&\qquad+\int_{\Om^\e}(\la\ov u^{-\de}+\ov u^{r-1}-\la u^{-\de}-u^{r-1})\phi^\e.
     \end{align*}
   Thus,
   \begin{align*}
   	\frac{1}{\e}E^\e\ge& \int_{\Om^\e}\big(|\na u|^{p-2}\na u-\ba|\na u|^{q-2}\na u-|\na\ov u|^{p-2}\na\ov u+\ba|\na\ov u|^{q-2}\na\ov u\big)\na\phi \\
   	&\quad-\la\int_{\Om^\e}|\ov u^{-\de}-u^{-\de}||\phi| \\
   	 &=o(1), \ \mbox{as }\e\ra 0,
   \end{align*}
   since $|\Om^\e|\ra 0$ as $\e\ra 0$. An analogous argument shows that
     \begin{align*}
     	\frac{1}{\e}E_\e\le o(1), \ \mbox{ as }\e\ra 0.
     \end{align*}
   Thus, from \eqref{eqb71} letting $\e\ra 0$, we obtain
    \begin{align*}
    		\int_{\Om}\big(|\na u|^{p-2}\na u\na\phi+\ba|\na u|^{q-2}\na u\na\phi-\la u^{-\de}\phi-u^{r-1}\phi\big)dx\ge 0.
    \end{align*}
  Since $\phi\in W^{1,p}_0(\Om)$ was arbitrary, so taking $-\phi$ in place of $\phi$, we get
  \begin{align*}
  \int_{\Om}\big(|\na u|^{p-2}\na u\na\phi+\ba|\na u|^{q-2}\na u\na\phi-\la u^{-\de}\phi-u^{r-1}\phi\big)dx= 0, \ \mbox{ for all } \phi\in W^{1,p}_0(\Om).
  \end{align*}
  \QED
\end{proof}

\begin{Lemma}\label{lemm17}
	For $\la\in(0,\La^*]$, $(P_\la)$ has a weak solution $u_\la$ in $W^{1,p}_0(\Om)$.
\end{Lemma}
\begin{proof}
	Fix $\la\in(0,\La^*)$. Let $\ul u_\la$ be the solution of the purely singular problem $(S_\la)$ (obtained in Lemma \ref{lemm10}). By definition of $\La^*$, there exists $\bar{\la}\in(\la,\La^*)$ such that $(P_{\bar\la})$ has a solution $u_{\bar\la}$. Then, by the weak formulations of $(P_{\bar\la})$ and $(S_\la)$, it is easy to see that $u_{\bar\la}$ is a supersolution and $\ul u_\la$ is a subsolution of $(P_\la)$. Applying Lemma \ref{lemm18} for $\bar{ u}= u_{\bar\la}$ and $\ul u_\la$, we get $\ul u_\la\le u_{\bar\la}$ a.e. in $\Om$.
	Then employing Lemma \ref{lemm12} for $\ul u=\ul u_\la$ and $\ov u= u_{\bar\la}$ when $\la\in(0, \La^*)$, we get a solution $u_\la$ of $(P_\la)$ such that $\ul u_\la\le u_\la\le u_{\bar\la}$. Moreover, by the fact that $u_\la$ is a minimizer of $I_\la$ on $M$, we deduce that $I_\la(u_\la)\le I_\la(\ul u_\la)\le\tilde{I_\la}(\ul u_\la)<0$.\\
	For $\la=\La^*$, let $\la_k\in(0, \La^*)$ be an increasing sequence such that $\la_k\ra\La^*$ and $u_k$ be the solution of $(P_{\la_k})$ obtained above. Moreover,
	\begin{align*}
	&I_{\la_k}(u_k)=\frac{1}{p}\int_\Om |\na u_k|^p+\ba\frac{1}{q}\int_{\Om}|\na u_k|^q -\frac{\la_k}{1-\de}\int_\Om |u_k|^{1-\de}~dx-\frac{1}{r}\int_\Om |u_k|^r~dx<0, \ \mbox{ and} \\
	&\|\na u_k\|_p^p+\ba\|\na u_k\|_q^q-\la_k\int_{\Om}|u_k|^{1-\de}~dx-\int_\Om |u_k|^r~dx=0,
	\end{align*}
	implies that $\{u_k\}$ is bounded in $W^{1,p}_0(\Om)$. Thus, there exists $u_{\La^*}\in W^{1,p}_0(\Om)$ such that $u_k\rightharpoonup u_{\La^*}$ weakly in $W^{1,p}_0(\Om)$ and $u_k(x)\ra u_{\La^*}(x)$ a.e. in $\Om$ (upto subsequence). By Lemma \ref{lemm18}, we have $u_{\La^*}\ge \ul u_{\la_1}> 0$ in $\Om$. Letting $k\ra \infty$ in the weak formulation of $(P_{\la_k})$ and using Lebesgue dominated convergence theorem, we get that $u_{\La^*}$ is a weak solution of $(P_{\La^*})$.\QED
\end{proof}
\textbf{Proof of Theorem \ref{thm5}:} Proof of the Theorem follows from Lemmas \ref{lemm16} and \ref{lemm17}.\QED

\subsection{Final remarks and perspectives}
(i) As it has been kindly suggested by one of the referees of this paper, we intend to continue and extend the analysis developed in this paper to singular double-phase problems with nonstandard growth of the type
$$\left\{\begin{array}{rllll}
   -{\rm div}\, (|\nabla u|^{p-2}+a(x)|\nabla u|^{p-2}\log (e+|\nabla u|)\nabla u) & = \la u^{-\de}+ u^{r-1} \quad  \text{ in } \Om \\   u&> 0 \quad \text{ in } \Om \\ u&=0 \quad \text{ on } \partial\Om.
\end{array}
\right.
$$

\smallskip
(ii) The same referee has also recommended to study further more general double-phase problems of the type
$$\left\{\begin{array}{rllll}
   -{\rm div}\, (|\nabla u|^{p-2}+a(x)|\nabla u|^{q-2}\nabla u) & = \la g(x,u)+ f(x,u) \quad  \text{ in } \Om \\  u&> 0 \quad  \text{ in } \Om \\ u&=0 \quad \text{ on } \partial\Om,
\end{array}
\right.
$$
where $0\leq a(\cdot)\in C^{0,\alpha}(\overline\Omega)$ and
$$\frac qp<1+\frac \alpha N\,.$$

The energy functional associated to this model contains the unbalanced variational integral
\begin{equation}\label{eenerg}u\mapsto \int_\Omega (|\nabla u|^p+a(x)|\nabla u|^q)dx.\end{equation}
The meaning of this functional is also to give a sharper version of the following energy
$$u\mapsto \int_\Omega |\nabla u|^{p(x)}dx,$$
thereby describing sharper phase transitions. 
In nonlinear elasticity and material science, composite materials with locally different hardening exponents $p$ and $q$ can be described using the energy defined in (\ref{eenerg}). Problems of this type are also motivated by applications to elasticity, homogenization, modelling of strongly anisotropic materials, Lavrentiev phenomenon, etc.

Accordingly, a new double phase model can be given by potentials of the form
$$\Phi_d(x,|\xi|):=\left\{
\begin{array}{lll}
& |\xi|^{p}+a(x)|\xi|^{q}&\quad\mbox{if}\ |\xi|\leq 1\\
& |\xi|^{p_1}+a(x)|\xi|^{q_1}&\quad\mbox{if}\ |\xi|\geq 1,
\end{array}\right.
$$
with $a(x)\geq 0$.

\smallskip
(iii) A new research direction corresponds to {\it anisotropic double-phase} operators with {\it singular} reaction. In this framework, we aim to develop the qualitative analysis performed in this paper to {\it singular} nonlinear boundary value problems with {\it variable exponents} of the type
\begin{equation}\label{aniso}\left\{\begin{array}{rllll}
   -{\rm div}\, \mathbf{A}(x,\nabla u) & = \la u^{-\de}+ u^{r-1}\\  u&> 0 \quad  \text{ in } \Om \\ u&=0 \quad \text{ on } \partial\Om,
\end{array}
\right.
\end{equation}
where
$$\mathbf{A}(x,\nabla
u)\mathbf{=}\left\{
\begin{array}{c}
\left\vert \nabla u\right\vert ^{p(x)-2}\nabla u\text{, if }\left\vert
\nabla u\right\vert >1 \\
\left\vert \nabla u\right\vert ^{q(x)-2}\nabla u\text{, if }\left\vert
\nabla u\right\vert \leq 1.
\end{array}%
\right. $$

This anisotropic model with unbalanced growth was introduced by Zhang and R\u adulescu \cite{zhang}. 

We conclude by pointing out that an important feature of nonlinear problems with {\it variable exponents} is that they can allow a ``subcritical-critical-supercritical" multiple regime, in the sense that $\Omega=\Omega_1\cup\Omega_2\cup\Omega_3$ and the problem is subcritical in $\Omega_1$, critical in $\Omega_2$ and supercritical in $\Omega_3$. We refer to Alves and R\u adulescu \cite{alves} for more details. A very interesting open problem corresponds to the analysis of the anisotropic singular case described by problem \eqref{aniso} in the multiple regime described above.

A related very interesting research direction corresponds to double-phase transonic flow problems  with variable growth driven by elliptic-hyperbolic  Baouendi-Grushin operators
with variable coefficients; see Bahrouni, R\u adulescu and Repov\v{s} \cite{bahr}.

\medskip
{\bf Acknowledgments.} The authors thank both anonymous referees for the careful reading of this paper and for their remarks  and  comments,  which  have   improved  the  initial  version  of  our  work.

\end{document}